\documentclass{article}
\usepackage{color}
\usepackage{amsmath}
\usepackage{amsfonts}
\usepackage{amssymb}
\usepackage{graphics}
\usepackage{comment}
\usepackage{url}

\newtheorem{theorem}{Theorem}
\newtheorem{lemma}[theorem]{Lemma}

\newtheorem{proposition}[theorem]{Proposition}

\newtheorem{corollary}[theorem]{Corollary}
\newtheorem{example}[theorem]{Example}

\def\R{\sf R}

\def\J{\sf J}
\def\pa{\mbox{\sl pa}}
\def\ppa{\mbox{\footnotesize\sl pa}}
\def\c{\mbox{\sf c}}
\def\p{\mbox{\sf p}}
\def\u{\mbox{\sf u}}
\def\us{\mbox{\small\sf u}}
\def\usc{\mbox{\scriptsize\sf u}}
\def\boldeta{\mbox{\boldmath$\eta$}}
\def\sboldeta{\mbox{\footnotesize\boldmath$\eta$}}
\def\boldA{\mbox{\boldmath$A$}}
\def\boldB{\mbox{\boldmath$B$}}
\def\boldC{\mbox{\boldmath$C$}}
\def\boldD{\mbox{\boldmath$D$}}
\def\boldE{\mbox{\boldmath$E$}}
\def\boldF{\mbox{\boldmath$F$}}
\def\boldI{\mbox{\boldmath$I$}}
\def\boldzero{\mbox{\boldmath$0$}}
\def\bolda{\mbox{\boldmath$a$}}
\def\boldas{\mbox{\scriptsize\boldmath$a$}}
\def\boldb{\mbox{\boldmath$b$}}
\def\boldc{\mbox{\boldmath$c$}}
\def\boldd{\mbox{\boldmath$d$}}
\def\bolde{\mbox{\boldmath$e$}}
\def\boldf{\mbox{\boldmath$f$}}
\def\boldx{\mbox{\boldmath$x$}}
\def\bepa{b} %\beta
\newenvironment{proof}{\begin{trivlist}\item[] \mbox{\it Proof. }}
{\hfill$\Box$ \end{trivlist}}

\def\ci{\perp\!\!\!\perp}
\def\dv{\mathbb}
\def\calA{{\cal A}}
\def\calB{{\cal B}}
\def\calC{{\cal C}}

\def\calI{{\cal I}}
\def\calK{{\cal K}}
\def\calL{{\cal L}}
\def\calP{{\cal P}(N)}

\definecolor{darkgreen}{RGB}{0,160,0}
\begin{document}
\title{On polyhedral approximations of polytopes for learning Bayes nets}

\author{Milan Studen\'{y}\footnote{Institute of Information Theory and Automation %,
 of the ASCR,
%Academy of Sciences of the Czech Republic,
studeny@utia.cas.cz} \and David Haws\footnote{University of Kentucky, Dept. of Statistics, dchaws@gmail.com}}

 \date{July 26, 2011}

\maketitle

 \begin{abstract}
 We review three vector encodings of Bayesian network structures.
 The first one has recently been applied by Jaakkola et al.\/
 \cite{Jaa10}, the other two use special integral vectors formerly introduced,
 called {\em imsets} \cite{Stu05,SHL10}. The
 central topic is the comparison of outer polyhedral approximations
 of the corresponding polytopes. We show how to transform the inequalities
 suggested by Jaakkola et al.\/ to the framework of imsets. The
 result of our comparison is the observation that the implicit
 polyhedral approximation of the standard imset polytope suggested
 in \cite{SV11} gives a closer approximation than the (transformed)
 explicit polyhedral approximation from \cite{Jaa10}. Finally, we
 confirm a conjecture from \cite{SV11} that the above-mentioned implicit
 polyhedral approximation of the standard imset polytope is an LP
 relaxation of the polytope.
 \end{abstract}

 \section{Introduction}

 {\em Bayesian networks\/} (BNs) are popular graphical statistical
 models widely used both in probabilistic reasoning \cite{Pea88} and
 statistics \cite{Lau96}. They are attributed to acyclic directed graphs
 whose nodes correspond to the variables in consideration. The motivation for
 this report is learning the BN structure \cite{Nea04} from data by maximizing a
 quality (= scoring) criterion. The criterion is a real function of a
 BN structure (= of a graph) and of a database; its value says how much
 the BN structure given by the graph is good to explain the occurrence of
 the database.

 However, different (acyclic directed) graphs can define the same
 statistical model, in which case the graphs are {\em Markov equivalent}. Thus,
 a usual requirement on the criterion is that it should be {\em score
 equivalent}, which means, it ascribes the same value to equivalent graphs.
 Another traditional technical requirement is that the criterion should be
 {\em decomposable\/} -- for details see \cite{Chi02}.

 Since the aim is learning the BN structure (= statistical model) some researchers
 prefer to have a unique representative for every BN structure and to
 understand the criterion as a function of such unique representatives.
 A traditional unique graphical representative of the BN structure is the
 {\em essential graph\/} of the corresponding Markov equivalence class
 of acyclic directed graphs, which is a special graph allowing both
 directed and undirected edges -- for details see \cite{AMP97}.

 The basic idea of an algebraic approach to learning, proposed
 in connection with conditional independence structures \cite{Stu05},
 is to represent every BN structure by a certain integral vector
 (= a vector with integers as components), called the {\em standard imset}.
 This is also a unique BN representative. The advantage of this %particular
 algebraic approach is that every score equivalent and decomposable criterion
 becomes an affine function of the standard imset.

 It has been shown in \cite{SVH10} that the standard imsets
 are vertices of a certain polytope, called the {\em standard imset
 polytope}. This allows one to re-formulate the learning task as a
 {\em linear programming\/} (LP) problem. However, to apply standard
 LP methods one needs the polyhedral description of the
 polytope. In \cite{SV11}, a conjecture about an implicit polyhedral
 characterization of the standard imset polytope has been presented. The
 weaker version of the conjecture was that the polyhedron given by
 those inequalities is an LP relaxation of the polytope.

 Suitable transformation of an LP problem often simplifies things.
 Therefore, in \cite{SHL10}, an alternative algebraic representative for
 the BN structure, called the {\em characteristic imset}, has been introduced.
 It is obtained from the standard imset by an invertible affine
 transformation; however, unlike the standard imset, the characteristic
 imset is always a zero-one vector. This opens the way to the application of
 advanced methods of {\em integer programming\/} (IP) in this area.
 Nonetheless, the crucial question of polyhedral characterization of the
 (trans\-for\-med) polytope remain to be answered.

 Jaakkola et al.\/ \cite{Jaa10} have also proposed to apply the methods of
 linear and integer programming to learning BN structures. They have used a
 straightforward zero-one encoding of acyclic directed graphs and transformed
 the task of maximizing the quality criterion to an IP problem. The
 main difference is that their vector codes are not unique BN
 representatives. On the other hand, they provide an explicit
 polyhedral LP relaxation of their polytope, which allows one to use
 the methods of IP.

 In this report, we transform the inequalities suggested by Jaakkola et al.\/
 to the framework of imsets. First, we show that the implicit
 polyhedral approximation of the standard imset polytope suggested
 in \cite{SV11} gives a closer approximation than the (transformed)
 explicit polyhedral approximation from \cite{Jaa10}. Second, we
 show that the transformed inequalities give an explicit LP relaxation
 of the standard/characteristic imset polytope. A consequence of
 this fact is the proof of the weaker version of the conjecture from \cite{SV11}.

 \section{Notation and terminology}

 Throughout the paper $N$ is a finite set of {\em variables\/}
 which has least two elements: $|N|\geq 2$. Its power set,
 denoted by $\calP$, is the class of its subsets $\{ A;\ A\subseteq N\}$.
 For any $\ell =1,2$, we use a special notation
 $$
 {\cal P}_{\ell}(N)\equiv \{A\subseteq N;\ |A|\geq \ell\}
 $$
 for the class of subsets of $N$ of cardinality at least $\ell$. The symbol
 $U\subset V$ will mean $U\subseteq V$, $U\neq V$.

 We deal with directed graphs (without loops) having $N$ as the set of nodes
 and call them {\em directed graphs over $N$}.
 Such a graph is specified by a collection of arrows $j\rightarrow i$,
 where $i,j\in N$, $i\neq j$; the set
 $\pa_{G}(i)\equiv\{ j\in N;\, j\rightarrow i\}$ is (called) the set
 of {\em parents\/} of node $i\in N$. A directed cycle in $G$ is
 a sequence of nodes $i_{1},\ldots, i_{n}$, $n\geq 3$ such that
 $i_{r}\rightarrow i_{r+1}$ in $G$ for $r=1,\ldots , n-1$ and
 $i_{n}=i_{1}$. A directed graph is {\em acyclic\/} if it has no
 directed cycle. A well-known equivalent definition is that
 there exists an ordering $i_{1},\ldots i_{|N|}$ of nodes of $G$
 {\em consistent with the direction of arrows\/} in $G$, which
 means $i_{r}\rightarrow i_{s}$ in $G$ implies $r<s$. Clearly, every
 acyclic directed graph $G$ has at least one {\em initial node},
 that is, a node $i$ with  $\pa_{G}(i)=\emptyset$.

 We also deal with real vectors, elements of ${\dv R}^{M}$, where $M$
 is a non-empty finite set.  By {\em lattice points}
 in ${\dv R}^{M}$ we mean integral vectors, that is,
 vectors whose components are integers (= elements of ${\dv Z}^{M}$).
 In this paper, $M$ has additional structure;
 typically, it is $\calP$ or ${\cal P}_{2}(N)$, in
 which cases the lattice points are called {\em imsets}. To
 write formulas for imsets we will use the following notation:
 given $A\subseteq N$, the corresponding {\em basic vector\/}
 will be denoted by $\delta_{A}$:
 $$
 \delta_{A}(S) = \left\{ \begin{array}{ll}
 1 & ~~\mbox{\rm if}~~ S=A\,,\\
 0 & ~~\mbox{\rm if}~~ S\subseteq N,\; S\neq A\,.
 \end{array} \right.
 $$
 A special {\em semi-elementary imset\/} $\u_{\langle A,B|C\rangle}$ is associated with any
 (ordered) triplet of pairwise disjoint sets $A,B,C\subseteq N$:
 $$
 \u_{\langle A,B|C\rangle} \equiv
 \delta_{C}-\delta_{A\cup C}-\delta_{B\cup C}+\delta_{A\cup B\cup C}\,,
 $$
 which, in the context of \cite{Stu05}, encodes the corresponding
 conditional independence statement $A\ci B\,|\,C$. The imsets
 will be denoted using sans serif fonts, e.g.\ $\u$ or $\c$;
 general vectors by bold lower-case letters, e.g.\ $\boldb$ or
 $\boldeta$. They are interpreted as column vectors.

 Matrices will be denoted by bold capitals, e.g.\ $\boldA$ or
 $\boldC$. The symbol $\boldA^{\top}$ denotes
 the transpose of $\boldA$. An invertible matrix $\boldA$ is {\em
 unimodular\/} if it is integral (= has integers as entries) and its determinant
 is $+1$ or $-1$ (see \S\,4.1 in \cite{Sch86}); an equivalent definition is
 that both $\boldA$ and its inverse $\boldA^{-1}$ are integral,
 that is, the mappings $\boldb\mapsto\boldA\boldb$ and
 $\boldc\mapsto{\boldA}^{-1}\boldc$ ascribe lattice points to lattice points.

 By a full row rank matrix we mean an $m\times n$-matrix which
 has $m$ linearly independent columns (= has rank $m$).
 The concept of unimodularity was extended in \S\,19.1 of
 \cite{Sch86} to matrices of this kind. A full row rank $m\times n$ matrix
 $\boldA$ is {\em unimodular\/} if every $m\times m$-submatrix
 has determinant $+1$, $-1$ or $0$; equivalently, if any of its
 invertible $m\times m$-submatrix $\boldB$ is unimodular.
 A matrix $\boldA$ is {\em totally unimodular\/} if any of its
 (square) submatrix has determinant $+1$, $0$ or $-1$.

 We also deal with special classes of subsets of $N$.
 More specifically, we will consider non-empty classes
 $\calA$ of non-empty subsets of $N$ which are {\em closed
 under supersets}. These are classes
 $\emptyset\neq\calA\subseteq{\cal P}_{1}(N)$ satisfying
 $$
 S\in\calA\,, ~~ S\subseteq T\subseteq N
 \quad\Rightarrow\quad T\in\calA\,.
 $$
 Every such class $\calA$ is characterized by the class $\calA_{\min}$
 of its {\em mimimal sets\/} with respect to inclusion:
 $$
 \calA_{\min}\equiv \{ S\in\calA ;\ \forall\,T\subset S\quad
 T\not\in\calA\}\,.
 $$
 Of course, $\calI =\calA_{\min}$ is a non-empty subclass of
 ${\cal P}_{1}(N)$ consisting of {\em incomparable sets}, which means
 $$
 \forall\, S,T\in\calI,\qquad  S\neq T ~~\Rightarrow ~~ [\,S\setminus T\neq\emptyset
 ~\,\&~\, T\setminus S\neq\emptyset\,]\,.
 $$
 Conversely, given a non-empty class $\calI\subseteq{\cal P}_{1}(N)$
 of incomparable sets the corresponding class $\calA$ closed under
 supersets satisfying $\calI =\calA_{\min}$ is as follows:
 $$
 \calA=\{S\subseteq N;\ \exists\,T\in\calI\quad T\subseteq S\}\,.
 $$

 Finally, in the proofs, we sometimes use Dirac's delta-symbol
 to shorten the notation. Specifically, the notation $\delta(\star\star)$, where
 $\star\star$ is a predicate (= statement), means a zero-one function whose
 value is $+1$ if the statement $\star\star$ is valid and whose
 value is $0$ if the statement $\star\star$ does not hold.

 \section{Three ways of encoding Bayes nets}

 \subsection{Straightforward zero-one encoding of a directed
 graph}\label{ssec.straight}

 Jaakkola et al.\/ \cite{Jaa10} used a special method for vector encoding
 (acyclic) directed graphs over $N$. Their 0-1-vectors \boldeta\ have components
 indexed by pairs $(i|B)$, where $i\in N$ and $B\subseteq N\setminus \{ i\}$.
 Although their intention was to encode acyclic directed graphs only, one can
 formally encode any directed graph in this way. Specifically,
 given a directed graph $G$ over $N$, the vector $\boldeta_{G}$ encoding $G$
 is defined as follows:
 $$
 \eta_{G}(i|B)=1 ~~\Leftrightarrow ~~ B=\pa_{G}(i), \qquad
 \eta_{G}(i|B)=0 ~~\mbox{otherwise}.
 $$

 \begin{example}\label{exa.1}\rm
 Consider $N=\{a,b,c\}$ and
 $G:a\leftrightarrows b\leftarrow c$. It is a directed graph,
 but not an acyclic one. We have $\pa_{G}(a)=\{ b\}$, $\pa_{G}(b)=\{a,c\}$,
 $\pa_{G}(c)=\emptyset$. Thus,  $\eta_{G}(a|\{b\})= 1$,
 $\eta_{G}(b|\{a,c\})= 1$, $\eta_{G}(c|\emptyset )= 1$, and
 $\eta_{G}(i|B)= 0$ otherwise.
 \end{example}

 The polytope studied by Jaakkola et. al.\/ \cite{Jaa10} is defined as the
 convex hull of the set of vectors $\boldeta_{G}$, where $G$ runs over all
 acyclic directed graphs over $N$.

 \subsubsection{Jaakkola et al.'s polyhedral
 approximation}\label{sssec.jaa-approx}
 The (outer) polyhedral approximation $\J$ of the above polytope proposed
 in \cite{Jaa10} is given by the following constraints:
 \begin{itemize}
 \item ``simple" {\em non-negativity constraints}:
 \begin{equation}
 \eta(i|B)\geq 0 \quad
 \mbox{\rm for every}~ i\in N,\, B\subseteq N\setminus\{i\}
 \label{eq.jaa-non-neg}
 \end{equation}
 ($|N|\cdot 2^{|N|-1}$ inequality constraints),
 \item {\em equality constraints}:
 \begin{equation}
 \sum_{B\subseteq N\setminus\{ j\}} \eta(j|B)=1 \quad
 \mbox{\rm for all}~ j\in N
 \label{eq.jaa-equal}
 \end{equation}
 ($|N|$ equality constraints),
 \item {\em cluster inequalities}, which correspond to sets
 $C\subseteq N$, $|C|\geq 2$:
 \begin{equation}
 1\leq \sum_{i\in C} ~\sum_{B\subseteq N\setminus\{ i\},\, B\cap C=\emptyset}
 \eta(i|B)
 %\stackrel{\mbox{\scriptsize for}\,\{ i\}\subseteq C}{\equiv}
 \equiv
 \sum_{i\in C} ~\sum_{D\subseteq N\setminus C} \eta(i|D)
 \label{eq.cluster-1}
 \end{equation}
 ($2^{|N|}-|N|-1$ cluster inequalities).
 \end{itemize}
 Taking into account the equality constraints (\ref{eq.jaa-equal})
 for $i\in C$, (\ref{eq.cluster-1}) takes the form
 $$
 1\leq
 \sum_{i\in C} ~\left[\,1- \sum_{B\subseteq N\setminus \{ i\},\, B\cap C\neq\emptyset}
 \eta(i|B)\,\right]\,.
 $$

 \noindent {\sl Remark} No cluster inequality for $C=\emptyset$ is
 defined; the cluster inequalities for $|C|=1$ are omitted because they follow
 trivially from the equality constraints.

 \begin{example}\label{exa.2}\rm
 In case $N=\{a,b,c\}$ every $\boldeta$-vector has length 12 and
 its components decompose into three blocks that correspond to variables $a$, $b$
 and $c$. Thus, one has twelve non-negativity constraints, three equality
 constraints and four cluster inequalities of two types:
 \begin{itemize}
 \item $1\leq \eta (a|\emptyset ) +\eta (a|\{ c\}) +\eta (b|\emptyset ) +\eta (b|\{
 c\})$, \hfill (for $C=\{ a,b\}$)
 \item $1\leq \eta (a|\emptyset ) +\eta (b|\emptyset ) +\eta (c|\emptyset
 )$. \hfill (for $C=\{ a,b,c\}$)
 \end{itemize}
 \end{example}

 The constraints (\ref{eq.jaa-non-neg}) and (\ref{eq.jaa-equal}) are
 clearly valid for any vector $\boldeta_{G}$ of a directed graph $G$;
 the inequalities (\ref{eq.cluster-1}) hold in the acyclic case --
 see Lemma \ref{lem.2}.

 \subsubsection{Jaakkola et al.'s approximation is an LP relaxation}\label{sssec.jaa-relax}

 The polyhedral approximation from \S\,\ref{sssec.jaa-approx} is
 an {\em LP relaxation\/} of the corresponding polytope, by which we mean that
 the only lattice points in the approximation are the lattice points in the polytope.
 First, we observe that the polyhedron  $\J^{'}$ given by non-negativity
 and equality constraints is an integral polytope.

 \begin{lemma}\label{lem.1}
 Let $\J^{\prime}$ be the polyhedron given by (\ref{eq.jaa-non-neg}) and (\ref{eq.jaa-equal}).
 Then $\J^{\prime}$ is a polytope whose vertices are just the codes of (general) directed
 graphs over $N$. Moreover, the only lattice points in $\J^{\prime}$ are its vertices.
 \end{lemma}

 \begin{proof}
 Let \boldeta\ belong to $\J^{\prime}$. For every block of components of $\boldeta$
 corresponding to $i\in N$, the constraints define a vector in a ``probability simplex".
 Assuming \boldeta\ is a vertex of $\J^{\prime}$, for each $i\in N$, the
 respective block has to be a vertex of that simplex, that is, a 0-1-vector having
 just one component 1. If $B(i)$ is the set indexing such a component for $i\in N$,
 we get the corresponding graph $G$ with $\boldeta =\boldeta_{G}$ by drawing arrows
 from the elements of $B(i)$ to $i$, for every $i\in N$.
 Clearly, this defines a one-to-correspondence between
 (general) directed graphs over $N$ and vertices of $\J^{\prime}$.

 Let $\boldeta$ be a lattice point in $\J^{\prime}$. Within the block
 given by $i\in N$, components are non-negative integers. Thus, if one
 of them exceeds 1, the sum exceeds 1. Hence, \boldeta\ is a 0-1-vector.
 At most one component in a block is 1 since otherwise the sum exceeds 1, and
 at least one is 1 since otherwise the sum is 0.
 \end{proof}

 \begin{lemma}\label{lem.2}
 Let $\J$ be the polyhedron given by constraints (\ref{eq.jaa-non-neg})-(\ref{eq.cluster-1}).
 Then the lattice points in $\J$ are exactly the codes of acyclic
 directed graphs over $N$.
 \end{lemma}

 \begin{proof}
 Every lattice point in $\J$ is a lattice point in $\J^{\prime}$, and,
 therefore, by Lemma \ref{lem.1}, encodes a (uniquely determined)
 directed graph $G$.

 Consider the cluster equality (\ref{eq.cluster-1}) for $C\subseteq N$,
 $|C|\geq 2$ and the vector $\boldeta_{G}$ (encoding a directed graph $G$).
 For every $i\in C$, the $\eta_{G}(i|D)$ term is typically 0 and only once 1, namely in
 the case $D=\pa_{G}(i)$. Thus, the inner expression for $i$ in
 (\ref{eq.cluster-1}), namely $\sum_{D\subseteq N\setminus C}
 \eta_{G}(i|D)$ is either 0 or 1. The latter happens if and only if
 $\pa_{G}(i)\cap C=\emptyset$. That means, the cluster inequality
 for $C$ says there exists at least one $i\in C$ with
 $\pa_{G}(i)\cap C=\emptyset$. Of course, this is true if $G$ is
 acyclic.

 Now, we are going to show the converse: the cluster inequalities
 for $\boldeta_{G}$ imply that $G$ is acyclic. We
 start with applying the cluster inequality for $C=N$ and
 find $i_{1}\in N$ with $\pa_{G}(i_{1})=\emptyset$. Thus, $i_{1}$ is an
 initial node in $G$ and we fix it. If $|N\setminus\{ i_{1}\}|\geq 2$
 we take $C=N\setminus\{ i_{1}\}$ and apply the cluster inequality
 for it. It says there exists $i_{2}\in C=N\setminus\{ i_{1}\}$
 with $\pa_{G}(i_{2})\cap C=\emptyset$, that is,
 $\pa_{G}(i_{2})\subseteq \{ i_{1}\}$ ($\equiv$ $i_{2}$ is the
 initial node in the induced subgraph $G_{N\setminus\{ i_{1}\}}$).

 Again, if $|N\setminus\{ i_{1},i_{2}\}|\geq 2$ we continue with
 $C=N\setminus\{ i_{1},i_{2}\}$, and so on. In this way, we find
 iteratively an ordering $i_{1},\ldots i_{|N|}$ consistent with the
 direction of arrows in $G$. This already implies $G$ is acyclic.
 \end{proof}

 \subsection{Standard imsets}

 Standard imsets introduced \cite{Stu05} have components indexed by subsets $T\subseteq N$.
 Given an acyclic directed graph $G$ over $N$, the {\em standard imset\/} $\u_{G}$
 encoding $G$ is defined as follows:
 $$
 \u_{G}= \delta_{N} - \delta_{\emptyset} +
 \sum_{i\in N}\,\left[\,\delta_{\ppa_{G}(i)} - \delta_{\{ i\}\cup\ppa_{G}(i)}\,\right].
 $$
 A basic property of standard imsets is that they are unique
 representatives of Bayesian network structures. This means, one has
 $\u_{G}=\u_{H}$ if and only if $G$ and $H$ are independence equivalent acyclic
 directed graphs (= define the same Bayesian network structure) --
 see Corollary 7.1 in \cite{Stu05}. In \cite{SVH10}, it was
 proposed to study the {\em standard imset polytope}, defined as
 the convex hull of the set of vectors $\u_{G}$, where $G$ runs over all
 acyclic directed graphs with $N$ vertices.

 \subsubsection{Outer approximation of the standard imset polytope}\label{sssec.our-approx}

  In \cite{SV11}, an outer approximation of the standard imset
  polytope in terms of linear constraints was suggested. More
  specifically, three types of constraints were considered
  (for $\u =\u_{G}$):
  \begin{itemize}
  \item {\em equality contraints}:
  \begin{equation}
  \sum_{T\subseteq N} \u(T)=0, \qquad
  \forall\, j\in N ~~  \sum_{T\subseteq N,\, j\in T} \u (T)=0\,,
  \label{eq.standardize}
  \end{equation}
  which implies that $\u$-vectors are determined uniquely by their
  components $\u (T)$ for $T\subseteq N$, $|T|\geq 2$,
  \item {\em specific inequality contraints} of the form:
  \begin{equation}
  \sum_{T\in {\cal A}} \u (T)\leq 1\,,
  \label{eq.specific-con}
  \end{equation}
  where ${\cal A}$ is a non-empty class of
  non-empty subsets of $N$, closed under supersets,
  \item {\em non-specific inequality contraints} of the form:
  \begin{equation}
  \langle m,\u\rangle\equiv \sum_{T\subseteq N} m(T)\cdot \u (T)\geq 0\,,
  \label{eq.non-specific-con}
  \end{equation}
  where $m$ is a (representative on an extreme
  standardized) supermodular function. Here, by a
  {\em supermodular function} is meant a real function
  $m$ on the power set ${\cal P}(N)$
  ($\equiv$ a vector in ${\dv R}^{{\cal P}(N)}$)
  such that
  $$
  m(E\cup F) + m(E\cap F)\geq m(E) + m(F)
  \quad \mbox{\rm for every $E,F\subseteq N$}.
  $$
  It is {\em standardized} if $m(T)=0$ whenever $|T|\leq 1$.
  \end{itemize}
  Note that the class of standardized supermodular functions on ${\cal P}(N)$ is
  a pointed rational polyhedral cone, and, therefore, has finitely
  many extreme rays. Each extreme ray contains a uniquely determined non-zero
  lattice point whose components have no common prime divisor (this
  is the representative of the extreme ray). Therefore, (\ref{eq.non-specific-con})
  gives in fact finitely many linear inequality constraints on $\u
  =\u_{G}$. The problem is that one has to compute those
  representatives of extreme supermodular functions, which is
  a difficult computational task. The representatives were computed
  for $|N|\leq 5$ \cite{SBK00}.

  Thus, in comparison with the polyhedral approximation (of the $\boldeta$-polytope)
  mentioned in \S\,\ref{sssec.jaa-approx}, this polyhedral
  approximation (of the standard imset polytope) is implicit.
  This is a disadvantage from the practical point of view because
  to apply common methods of linear programming one still needs to explicate
  the considered inequality constraints for any $|N|$.

 \begin{example}\label{exa.3}\rm
 In case $N=\{a,b,c\}$ every $\u$-vector has the length 8. There are four equality constraints
 (\ref{eq.standardize}) which break into two types:
 \begin{itemize}
 \item $\u (\emptyset )= -\u (a) -\u (b) -\u (c)
 -\u (\{ a,b\}) -\u (\{ a,c\}) -\u (\{ b,c\}) -\u (\{ a,b,c\})$,
 \item $\u (a)= -\u (\{ a,b\}) -\u (\{ a,c\})-\u (\{ a,b,c\})$.
 \hfill (for $j=a$)%
 \end{itemize}
 Therefore, the dimension (of the standard imset polytope) is
 4 and the $\u$-vectors are determined by their components for sets $\{
 a,b\}$, $\{ a,c\}$, $\{ b,c\}$ and $\{ a,b,c\}$.

 As concerns specific inequality constraints, every non-empty class of ${\cal A}$
 of non-empty subsets of $N$ closed under supersets is uniquely
 determined by the class ${\cal A}_{\min}$ of its minimal sets with
 respect to inclusion. One has eighteen such classes which break into eight
 types. For example, ${\cal A}_{\min}=\{ ab,ac,bc\}$ gives the inequality
 $$
 \u (\{ a,b\})  +\u (\{ a,c\}) +\u (\{ b,c\}) +\u (\{ a,b,c\})\leq
 1\,.
 $$
 As concerns non-specific inequality constraints, the cone of
 standardized supermodular functions has five extreme rays in case
 $|N|=3$ \cite{SBK00},
 which leads to five inequalities breaking into three types:
 \begin{itemize}
 \item $\u (\{ a,b,c\})\geq 0$,
 \item $\u (\{ a,b\}) +\u (\{ a,b,c\})\geq 0$,
 \item $\u (\{ a,b\}) +\u (\{ a,c\}) +\u (\{ b,c\}) +2\cdot\u (\{ a,b,c\})\geq 0$.
 \end{itemize}
 \end{example}

 Note that the described system of inequalities can be reduced; some
 of the specific inequalities appear to follow from the non-specific
 ones in combination with equality constraints and other specific
 inequalities. For example, if ${\calA}_{\min}$ consists of one
 singleton only, then the respective specific inequality (\ref{eq.specific-con})
 is vacuous because it trivially follows from the equality
 constrains (\ref{eq.standardize}). Actually, all specific inequalities with
 ${\calA}_{\min}$ containing a singleton are superfluous in case
 $|N|=3$. However, this is not true in case $|N|\geq 4$.

  The constraints (\ref{eq.standardize})-(\ref{eq.non-specific-con})
  were conjectured in \cite{SV11} to completely characterize the
  standard imset polytope and this conjecture was verified for $|N|\leq 4$.
  Nevertheless, one perhaps does not need a complete facet description
  (= polyhedral characterization) of the polytope. To apply some
  advanced methods of integer programming the confirmation of a
  weaker version of the conjecture might be enough. The weaker
  version of the conjecture from \cite{SV11} is that the
  polyhedron given by (\ref{eq.standardize})-(\ref{eq.non-specific-con})
  is an LP relaxation of the standard imset polytope.

  Before writing this report, we confirmed computationally the weaker version for
  $|N|=5$. The extreme rays of the cone of supermodular functions for $|N|=5$
  were obtained from \cite{SBK00} and independently computed using {\sl 4ti2}
  \cite{4ti2}, thus giving the non-specific inequality constraints
  (\ref{eq.non-specific-con}). Specific inequality constraints (\ref{eq.specific-con})
  were obtained from \cite{SV11}, where it was also calculated that there are $8,782$
  standard imsets for $|N|=5$.  Since the characteristic imsets (described in
  \S\,\ref{sec:charimset}) are 0-1-vectors and are in one-to-one correspondence to the
  standard imsets, we simply enumerated all vectors in ${\{ 0,1\}}^{{\cal P}_{2}(N)}$,
  applied the inverse transform (\ref{eq.inverse}) to get the corresponding $\u$-vectors,
  and tested whether they satisfied the above inequalities. By operating over ${\cal P}_2(N)$, and
  properly modifying the above inequalities, the equality constraints (\ref{eq.standardize})
  were satisfied. We verified that there were exactly $8,782$ integer solutions to the constraints
  (\ref{eq.standardize})-(\ref{eq.non-specific-con}) for $|N|=5$.

 \subsubsection{$\boldeta$ to standard imset}\label{sssec.eta-to-stan}
 Taking into account the definition of
 $\boldeta_{G}$, it is easy to see that $\u_{G}$ is obtained from
 $\boldeta_{G}$ by applying the following mapping
 $\boldeta\mapsto \u^{\sboldeta}$. For any $T\subseteq N$, we put
 \begin{equation}
 \u^{\sboldeta}(T) = \delta_{N}(T) - \delta_{\emptyset}(T) +
 \sum_{i\in N} \sum_{B\subseteq N\setminus \{ i\}} \eta(i|B)
 \cdot \{\delta_{B}(T)-\delta_{\{i\}\cup B}(T) \}\,.
 \label{eq.jaa-to-stand}
 \end{equation}
 This is clearly an affine mapping, ascribing lattice points to
 lattice points.
 Assuming $\boldeta$ belongs to the linear subspace
 specified by equality constraints (\ref{eq.jaa-equal}), we
 re-write (\ref{eq.jaa-to-stand}) as follows:
 \small
 \begin{eqnarray*}
 \lefteqn{\us^{\sboldeta}(T) =
 \delta_{N}(T) - \delta_{\emptyset}(T)} \\
 &+& \sum_{i\in N}\, \eta(i|\emptyset)
 \cdot\{\delta_{\emptyset}(T)-\delta_{\{i\}}(T)  \} +
 \sum_{i\in N} \sum_{\emptyset\neq B\subseteq N\setminus \{ i\}} \eta(i|B)
 \cdot \{\delta_{B}(T)-\delta_{\{i\}\cup B}(T) \}\\
 &\stackrel{(\ref{eq.jaa-equal})}{=}& \delta_{N}(T) - \delta_{\emptyset}(T) + \sum_{i\in N}\,
 \{ 1- \sum_{\emptyset\neq B\subseteq N\setminus \{ i\}} \eta(i|B)\}
 \cdot\{\delta_{\emptyset}(T)-\delta_{\{i\}}(T)\} + \ldots \\
 &=& \underbrace{\delta_{N}(T) +(|N|-1)\cdot \delta_{\emptyset}(T)
 -\sum_{i\in N} \delta_{\{i\}}(T) }_{\usc^{\emptyset}(T)\in {\dv Z}} \\
 &-& \sum_{i\in N} \sum_{\emptyset\neq B\subseteq N\setminus \{ i\}}
 \eta(i|B) \cdot \underbrace{\{\delta_{\emptyset}(T) -\delta_{\{i\}}(T)
 -\delta_{B}(T) +\delta_{\{i\}\cup B}(T)
 \}}_{\usc_{\langle i,B|\emptyset\rangle}(T)\in \{-1,0,+1\}}\,\,,
 \end{eqnarray*}
 \normalsize
 where $\u^{\emptyset}$ denotes the standard imset corresponding
 to the empty graph over $N$ and $\u_{\langle i,B|\emptyset\rangle}$
 the semi-elementary imset encoding $i\ci B\,|\,\emptyset$.

 Briefly, if $\boldeta$ satisfies (\ref{eq.jaa-equal}) then
 $$
 \u^{\sboldeta}= \u^{\emptyset} - \sum_{i\in N}\,
 \sum_{\emptyset\neq B\subseteq N\setminus \{ i\}} \eta(i|B)\cdot
 \u_{\langle i,B|\emptyset\rangle}.
 $$
 In particular, $\u=\u^{\sboldeta}$ belongs to the linear subspace
 specified by equality constraints (\ref{eq.standardize}).
 This is because these equalities hold for both $\u^{\emptyset}$
 and any $\u_{\langle i,B|\emptyset\rangle}$. Note that the converse
 is true as well (we leave an easy proof to the reader): if $\u$
 satisfies (\ref{eq.standardize}) then there exists $\boldeta$
 satisfying (\ref{eq.jaa-equal}) such that $\u=\u^{\sboldeta}$.
 In particular, (\ref{eq.standardize}) is the exact translation of
 (\ref{eq.jaa-equal}) into the framework of standard imsets.

 \subsection{Characteristic imsets}
 \label{sec:charimset}

 The characteristic imset (for an acyclic directed graph $G$), introduced in
 \cite{SHL10},
 is obtained from the standard imset by an affine transformation. More specifically,
 first, the {\em portrait} $\p_{G}$ of the standard imset $\u_{G}$ is
 obtained by a linear transform; second, the portrait is subtracted
 from the constant $1$-vector and the {\em characteristic imset} $\c_{G}$ is obtained:
 \begin{eqnarray}
 \p (S) &=& \sum_{T,\, S\subseteq T\subseteq N} \u (T) \qquad
 \mbox{for $S\subseteq N$},\label{eq.portrait}\\
 \c (S) &=& 1-\p (S) \qquad \mbox{for $S\subseteq N$}. \label{eq.characteristic}
 \end{eqnarray}
 Clearly, the equality constraints (\ref{eq.standardize}) are translated
 into the following tacit restrictions on $\c$-vectors:
 \begin{equation}
 \c (S)=1 \qquad \mbox{for $S\subseteq N$, $|S|\leq 1$}\,.
 \label{eq.equal-char}
 \end{equation}
 Therefore, for an acyclic directed graph $G$ over $N$, the components of
 the characteristic imset $\c_{G}$ for $|S|\leq 1$ are ignored and $\c_{G}$
 is formally considered to be an element of ${\dv Z}^{{\cal P}_{2}(N)}$.

 The mapping $\u\mapsto\c$ determined by (\ref{eq.portrait})-(\ref{eq.characteristic})
 is invertible: one can compute back the standard imset by the formula
 \begin{equation}
 \u (T)=\sum_{S,\, T\subseteq S\subseteq N} (-1)^{|S\setminus T|}\cdot
 \underbrace{[\,1-\c (S)\,]}_{\p (S)} ~~\qquad
 \mbox{for $T\subseteq N$}.
 \label{eq.inverse}
 \end{equation}
 Indeed, to see it fix $S\subseteq N$, substitute (\ref{eq.inverse})
 (with $S$ replaced by $D$) into the expression for the portrait $\p (S)$ and
 change the order of summation:
 \begin{eqnarray*}
 \sum_{T,\, S\subseteq T\subseteq N} \u(T) &=&
 \sum_{T,\, S\subseteq T\subseteq N} \ \sum_{D,\, T\subseteq D\subseteq N}
 (-1)^{|D\setminus T|}\cdot  \p(D)\\
 &=& \sum_{D,\, S\subseteq D\subseteq N}  \p(D)\cdot
 \underbrace{\sum_{T,\, S\subseteq T\subseteq D}
 (-1)^{|D\setminus T|}}_{\delta_{S}(D)} =\p(S)\,.
 \end{eqnarray*}
 Since the transformation is one-to-one, two acyclic
 directed graph $G$ and $H$ are independence equivalent if and only if $\c_{G}=\c_{H}$.
 Thus, the characteristic imset is also a unique Bayesian network structure
 representative.

 \subsubsection{Advantage of characteristic imsets}\label{sssec.advantage}
 Since standard and characteristic imsets are in one-to-one correspondence,
 one can transform the inequality constraints from \S\,\ref{sssec.our-approx}
 into the framework of characteristic imsets -- see \S\,\ref{sssec.trans-elem-char} and
 \S\,\ref{ssec.trans-cluster} for further details. One important consequence
 of these transformed constraints are {\em basic inequalities for
 characteristic imsets\/} valid in the acyclic case:

 \begin{corollary}\label{cor.basic-char}
 The constraints (\ref{eq.standardize})-(\ref{eq.non-specific-con}) on $\u$
 imply the inequalities $0\leq \c (S)\leq 1$, $S\subseteq N$ for the
 imset $\c$ ascribed to $\u$ by (\ref{eq.portrait})-(\ref{eq.characteristic}).
 \end{corollary}

 \begin{proof}
 Because of  (\ref{eq.characteristic}), we show $0\leq \p (S)\leq 1$ for
 $S\subseteq N$. First, (\ref{eq.standardize}) says $\p (S)=0$ for $|S|\leq 1$.
 Given $S\subseteq N$, $|S|\geq 2$ the class of sets
 $\calA =\{ T;\ S\subseteq T\subseteq N\}$ is closed under supersets and,
 by (\ref{eq.specific-con}), $\p (S)\leq 1$. On the other hand,
 in (\ref{eq.non-specific-con}), among the (representatives of extreme)
 supermodular functions we find the function
 $$
 m^{S\uparrow}(T)=\left\{\,
 \begin{array}{lll}
 1 &\quad & \mbox{if $S\subseteq T$,} \\
 0 && \mbox{otherwise}.
 \end{array}
 \right.
 $$
 In particular, among the non-specific inequality constraints is the
 inequality
 $\p (S)=\sum_{T,\, S\subseteq T} \u (T)\equiv\langle m^{S\uparrow},\u\rangle\geq 0$.
 \end{proof}

 In particular, every characteristic imset $\c_{G}$ (for an acyclic directed graph $G$)
 is a 0-1-vector, which is a fact emphasized already in \cite{SHL10}, which is important from the
 point of view of (possible future application of) methods of integer programming.

 Another advantage of characteristic imsets is that they are closer
 to the graphical description (of Bayesian network structures) than standard imsets.
 Specifically, for $S\subseteq N$, $|S|\geq 2$ one has
 \begin{equation}
 \c_{G}(S)=1 ~~ \Leftrightarrow \mbox{ there exists $i\in S$ with $S\setminus \{
 i\}\subseteq\pa_{G}(i)$},
 \label{eq.char.imset.acyc}
 \end{equation}
 and there exists a polynomial algorithm for transforming the characteristic
 imset $\c_{G}$ into the respective {\em essential graph}, which is
 a traditional unique graphical representative of the Bayesian network structure
 given by $G$ -- see \cite{SHL10}.

 \subsubsection{$\boldeta$ to characteristic imset}\label{sssec.eta-to-char}

 \begin{lemma}\label{lem.3}
 The characteristic imset $\c_{G}$ is a linear function of $\boldeta_{G}$
 given by
 \begin{equation}
 \c (S) = \sum_{i \in S}\ \sum_{B,\, S\setminus\{ i\}\subseteq B\subseteq N\setminus\{ i\}}
 \eta(i|B)
 \qquad \mbox{where $|S| \geq 1$. }
 \label{eq.eta-to-char}
 \end{equation}
 \end{lemma}

 \begin{proof}
 Given  $S\subseteq N$, substitute (\ref{eq.jaa-to-stand}) into
 (\ref{eq.portrait}) and change the order of summation:
 \small
 \begin{eqnarray*}
 \p (S) &=&  \sum_{T,\, S\subseteq T\subseteq N}
 \big[\,\delta_{N}(T) - \delta_{\emptyset}(T) +
 \sum_{i\in N} \sum_{B\subseteq N\setminus \{ i\}} \eta(i|B)
 \cdot \{\delta_{B}(T)-\delta_{\{i\}\cup B}(T) \}\,\big]\\
 &=&  \sum_{T,\, S\subseteq T\subseteq N} \delta_{N}(T) -
 \sum_{T,\, S\subseteq T\subseteq N}\delta_{\emptyset}(T)\\
 &&
 +\sum_{i\in N} \sum_{B\subseteq N\setminus \{ i\}} \eta(i|B)\cdot
 \big\{\, \sum_{T,\, S\subseteq T\subseteq N}\delta_{B}(T)-
 \sum_{T,\, S\subseteq T\subseteq N}\delta_{\{i\}\cup B}(T) \,\big\}\\
 &=& 1-\delta_{\emptyset}(S) +
 \sum_{i\in N} \sum_{B\subseteq N\setminus \{ i\}} \eta(i|B)\cdot
 \big\{\, \delta (S\subseteq B)- \delta (S\subseteq \{ i\}\cup B)
 \,\big\}\,.
 \end{eqnarray*}
 \normalsize
 Realize that the expression $\delta (S\subseteq B)- \delta (S\subseteq \{ i\}\cup B)$
 vanishes if either $S\subseteq B$ or $S\setminus (\{ i\}\cup B)\neq\emptyset$, otherwise it is
$-1$. Thus, assuming $|S|\geq 1$, one has
 \begin{eqnarray*}
 \p (S) &=&  1+ \sum_{i\in N} \sum_{B\subseteq N\setminus \{ i\}} \eta(i|B)\cdot
 (-1)\cdot \delta (i\in S,\ S\subseteq \{ i\}\cup B)\\
 &=& 1- \sum_{i\in S} \sum_{B\subseteq N\setminus \{ i\}}
\eta(i|B)\cdot
 \delta (S\setminus\{ i\}\subseteq B)\,,
 \end{eqnarray*}
 because, in case $i\in S$, then $S\subseteq \{ i\}\cup B$ is equivalent to $S\setminus\{ i\}\subseteq
 B$ . Taking (\ref{eq.characteristic}) into
 consideration we get (\ref{eq.eta-to-char}).
 \end{proof}

 Let us call the mapping given by (\ref{eq.eta-to-char})
 the {\em characteristic transformation}. It can formally be
 applied to any $\boldeta$-vector, in particular, to the code $\boldeta_{G}$ of a general
 directed graph $G$. Thus, we get a formula for the
 ``quasi-characteristic" imset (= an element of ${\dv Z}^{{\cal P}_{2}(N)}$)
 ascribed to a graph over $N$:
 \begin{equation}
 \c_{G}(S)=\mbox{number of super-terminal nodes in $S$}
 \quad \mbox{for $S\subseteq N$, $|S|\geq 2$}.
 \label{eq.quasi-char}
 \end{equation}
 Here, a {\em super-terminal node} (in $S$) means $i\in S$ such that for all $j\in
 S\setminus \{i\}$ one has $j\rightarrow i$ in $G$.
 Indeed, having fixed $S$, $|S|\geq 2$ and $i\in S$, the expression
 $\sum_{B,\, S\setminus\{ i\}\subseteq B\subseteq N\setminus\{ i\}} \eta_{G}(i|B)$
 is either $0$ or $1$ depending upon $S\setminus\{ i\}\subseteq \pa_{G}(i)$.
 Observe that (\ref{eq.char.imset.acyc}) is a special case (\ref{eq.quasi-char})
 since, in case of an acyclic directed graph, any set $S$ has at most one super-terminal node.

 \begin{example}\label{exa.4}\rm
 Consider the graph $G$ from Example \ref{exa.1}. Then $\c_{G}(\{ a,c\})=0$,
 $\c_{G}(\{ b,c\})=\c_{G}(\{ a,b,c\})=1$ and $\c_{G}(\{ a,b\})=2$.
 Observe that $\c_{G}$ does not satisfy the basic constrains $0\leq \c \leq 1$
 valid in acyclic case. This is because $G$ is not acyclic.
 \end{example}

 \section{Transformation of inequality constraints}
 In \S\,\ref{sssec.eta-to-stan} and \S\,\ref{sssec.eta-to-char}, we
 have described mappings which transform the $\boldeta$-vectors used by
 Jaakkola et al.\/ \cite{Jaa10} to standard/characteristic imsets.
 The advantage of the $\boldeta$-polytope is the existence of a good
 (= explicit) outer polyhedral approximation (see Lemma \ref{lem.2}
 in \S\,\ref{sssec.jaa-approx}). In this section, we
 characterize the image of that polyhedral approximation (by the above
 maps) and compare the transformed approximation (of $\boldeta$-polytope)
 with the approximation of the standard imset polytope from \S\,\ref{sssec.our-approx}.
 The main technical difficulty we have to tackle is that the
 mappings transforming $\boldeta$-vectors to imsets are many-to-one.
 Another feature is that the transformation raises the number of
 linear constraints. To clarify the reasons for that, in
 \S\,\ref{ssec.trans-elem} we first deal with the transformation of elementary
 constraints (\ref{eq.jaa-non-neg})-(\ref{eq.jaa-equal}) and,
 later, in \S\,\ref{ssec.trans-cluster}, with the transformation of cluster
 inequalities (\ref{eq.cluster-1}).

 \subsection{Transformation of elementary $\boldeta$-constraints}\label{ssec.trans-elem}
 Now, the question of our interest is to transform the
 constraints (\ref{eq.jaa-non-neg})-(\ref{eq.jaa-equal}) only, that is,
 to characterize the form of the inequalities of the image of the polyhedron
 $\J^{\prime}$ from Lemma \ref{lem.1}. Let us start with an example,
 illustrating our method.

 \begin{example}\label{exa.5}\rm
 Consider $N=\{a,b,c\}$, the polyhedron $\J^{\prime}$ and the characteristic
 transformation $\boldeta\mapsto\c$ given by (\ref{eq.eta-to-char}).
 The idea is to transform each vertex of $\J^{\prime}$ and take the
 convex hull $\R$ of the images of vertices. Because of linearity of the
 map $\boldeta\mapsto\c$, the polytope $\R$ is the image of $\J^{\prime}$.
 Thus, it is enough to find the facet description of $\R$; this is
 the exact translation of (\ref{eq.jaa-non-neg})-(\ref{eq.jaa-equal}) then.

 The vertices of  $\J^{\prime}$ are exactly the codes of general directed graphs
 (see Lemma \ref{lem.1}) and their images are given by (\ref{eq.quasi-char}).
 Thus, the (permutation type representatives of) images of vertices of $\J^{\prime}$
 were obtained in this way. Here they are (the order of component is $ab,ac,bc,abc$):
 \begin{eqnarray*}
 &[0, 0, 0, 0],
 [1, 0, 0, 0],
 [2, 0, 0, 0],
 [2, 1, 0, 0],
 [1, 1, 0, 0],
 [1, 1, 1, 0],&\\
 &[1, 1, 0, 1],
 [2, 1, 0, 1],
 [2, 2, 0, 1],
 [1, 1, 1, 1],
 [2, 1, 1, 1],&\\
 &[2, 1, 1, 2],
 [2, 2, 1, 2],
 [2, 2, 2, 3].&
 \end{eqnarray*}
 Remaining images can be obtained by permutation of first 3
 components. We computed the facet-description of their convex hull $\R$ by
 {\sl Polymake} \cite{Polymake}. The result had fifteen inequalities.
 Here, we only recorded the (permutation) types of obtained inequalities:
 \begin{itemize}
 \item $0\leq \c(ab)$,
 \item $0\leq 2-\c(ab)$,
 \item $0\leq 3-\c(ab)-\c(ac)-\c(bc)+\c(abc)$,
 \item $0\leq \c(abc)$,
 \item $0\leq 1+\c(ab)-\c(abc)$,
 \item $0\leq \c(ab)+\c(ac)-\c(abc)$,
 \item $0\leq \c(ab)+\c(ac)+\c(bc)-2\c(abc)$.
 \end{itemize}
 To make sure we computed the vertices of the polyhedron given by
 these inequalities. The (permutation) type representatives are as
 follows:
 $$
 [0,0,0,0],
 [2,0,0,0],
 [2,1,0,0],
 [1,1,0,1],
 [2,1,0,1],
 [2,2,0,1],
 [2,1,1,2],
 [2,2,2,3].
 $$
 We observe that some of images of vertices of $\J^{\prime}$ are convex combinations of the
 others: for example, $[1,0,0,0]$ comes from $[0,0,0,0]$ and $[2,0,0,0]$.
 Note that the original polyhedron $\J^{\prime}$ was given by twelve inequalities
 (and three equality constraints). Since $\R$ is given by fifteen inequality (and four
 implicit equality) constraints, the transformation to the framework of
 characteristic imsets raised the number of inequality constraints.

 Another interesting observation is that the obtained fifteen inequalities in fact coincide
 with the translation of specific inequality constraints (\ref{eq.specific-con})
 to the framework of characteristic imsets in case $N=\{ a,b,c\}$ -- see
 Example \ref{exa.6} for details.
 \end{example}

 This leads to a natural conjecture that Jaakkola et al.'s elementary constraints
 (\ref{eq.jaa-non-neg})-(\ref{eq.jaa-equal}) are equivalent to our specific constraints
 for any $|N|$. We confirm this conjecture below, directly by considering the
 transformation of $\boldeta\mapsto\u$. Later, we transform the specific constraints
 to the framework of characteristic imsets (see \S\,\ref{sssec.trans-elem-char}).

 \subsubsection{Translation to the framework of
 standard imsets}\label{sssec.trans-elem-stand}

 Thus, the task is to characterize in terms of $\u$ the image (by
 $\boldeta\mapsto \u^{\sboldeta}$) of the polytope
 $\J^{\prime}$ given by non-negativity and equality constraints.
 More specifically, we wish to have a finite system of linear
 inequalities on $\u$ which together with (\ref{eq.standardize})
 -- see \S\,\ref{sssec.eta-to-stan} -- characterize those
 $\u\in {\dv R}^{{\cal P}(N)}$ for which
 \begin{equation}
 \exists\, \boldeta
 ~~ \mbox{satisfying (\ref{eq.jaa-non-neg}),(\ref{eq.jaa-equal})
 and}~ \u^{\sboldeta}(T)=\u(T) ~~\mbox{for any}~ T\subseteq N,|T|\geq 2\,.
 \label{eq.problem-1}
 \end{equation}
 This task can equivalently be formulated as follows. Let us put
 $m\equiv 2^{|N|}-1$, $n\equiv |N|\cdot 2^{|N|-1}$ and consider
 a special $m\times n$ matrix $\boldA$, whose
 \begin{itemize}
 \item rows correspond to sets $T\subseteq N$, $|T|\geq 1$,
 \item columns correspond to pairs $(i|B)$ where $i\in N$,
 $B\subseteq N\setminus\{ i\}$.
 \end{itemize}
 More specifically, the entry $\bolda\, [\,T,(i|B)\,]$ of $\boldA$ is
 given by
 \begin{equation}
 \begin{array}{lcll}
 \bolda\, [\,T,(i|B)\,] &=& \delta_{\{i\}\cup B}(T)-\delta_{B}(T)
 \qquad &\mbox{if}~~ |T|\geq 2\,,\\[0.3ex]
 \bolda\, [\,T,(i|B)\,] &=& \delta_{\{i\}}(T)
 &\mbox{if}~~ |T|=1\,.
 \end{array}
 \label{eq.def-A-matrix}
 \end{equation}
 Moreover, to any $\u\in {\dv R}^{{\cal P}_{2}(N)}$, we ascribe
 a column $m$-vector $\boldb_{\usc}$ whose components $\bepa_{\usc}\,[\,T\,]$
 are specified as follows:
 \begin{eqnarray*}
 \bepa_{\usc}\, [\,T\,] &=& \delta_{N}(T)-\u(T)
 \quad ~ \mbox{if}~~ |T|\geq 2\,,\\
 \bepa_{\usc}\, [\,T\,] &=& 1
 \qquad\qquad\qquad\quad \mbox{if}~~ |T|=1\,.
 \end{eqnarray*}
 Then (\ref{eq.problem-1}) is equivalent to the condition
 \begin{equation}
 \exists\, \boldeta\in {\dv R}^{n}
 ~~ \mbox{satisfying $\boldeta\geq 0$ and $\boldA\boldeta=\boldb_{\usc}$.}
 \label{eq.problem-2}
 \end{equation}
 {\small\sl
 Indeed, (\ref{eq.jaa-non-neg}) means $\boldeta\geq 0$, while
 (\ref{eq.jaa-equal}) for $j\in N$ is the requirement that
 the component of $\boldb_{\usc}$ for $T=\{ j\}$, which is $1$,
 coincides with the respective component of $\boldA\boldeta$:
 $$
 1=\sum_{(i|B)}
 \bolda\, [\,T,(i|B)\,]\cdot \eta(i|B) =
 \sum_{i\in N}\sum_{B\subseteq N\setminus\{ i\}}
 \delta_{\{i\}}(\{ j\})\cdot \eta(i|B) =
 \sum_{B\subseteq N\setminus\{ j\}} \eta(j|B)\,.
 $$
 Analogously, for fixed $T\subseteq N$, $|T|\geq 2$, $\u(T)=\u^{\sboldeta}(T)$ has,
 by (\ref{eq.jaa-to-stand}), the form
 $$
 \us(T)=\delta_{N}(T)-\sum_{i\in N}\sum_{B\subseteq N\setminus\{ i\}}
 \underbrace{\{\delta_{\{i\}\cup B}(T)-
 \delta_{B}(T)\}}_{\boldas\, [\,T,(i|B)\,]}\cdot\, \eta(i|B)
 $$
 and can be expressed equivalently as
 $$
 \sum_{i\in N}\sum_{B\subseteq N\setminus\{ i\}}
 \bolda\, [\,T,(i|B)\,]\cdot \eta(i|B) =\delta_{N}(T)-\us (T)
 \equiv \bepa_{\usc}(T)\,,
 $$
 which means the components of $\boldA\boldeta$ and $\boldb_{\usc}$ for
 $T$ coincide.
 }

 Now, Farkas' lemma (see Corollary 7.1d in \cite{Sch86}) applied to
 $\boldA$ and $\boldb_{\usc}$ says that (\ref{eq.problem-2}) is
 equivalent to the requirement:
 \begin{equation}
 \forall\, y\in {\dv R}^{m} \quad
 \boldA^{\top}y\geq 0 ~~\Rightarrow ~~
 \boldb_{\usc}^{\top}y\geq 0\,.
 \label{eq.problem-3}
 \end{equation}
 To simplify this requirement we re-write the condition
 $\boldA^{\top}y\geq 0$ in this form:
 \begin{eqnarray}
 &\forall\, i\in N & ~~y(\{i\})\geq 0, \label{eq.a1}\\
 & \forall\, S\subseteq N,|S|=2, ~\forall\,i\in S &
 ~~y(S)+y(\{i\})\geq 0, \label{eq.a2}\\
 &\forall\, S\subseteq N,|S|\geq 3, ~\forall\,i\in S &
 ~~y(S)+y(\{i\})-y(S\setminus\{ i\})\geq 0. \label{eq.a3}
 \end{eqnarray}
 {\small\sl
 Indeed, the rows of $\boldA^{\top}$ correspond to pairs
 $(i|B)$, $i\in N$, $B\subseteq N\setminus\{ i\}$. If $i\in N$ and
 $B=\emptyset$ then the component of $\boldA^{\top}y$
 for $(i|\emptyset )$ is as follows:
 $$
 \sum_{\emptyset\neq T\subseteq N}
 \bolda\, [\,T,(i|\emptyset )\,]\cdot y(T)=
 \sum_{|T|=1} \delta_{\{i\}}(T)\cdot y(T)= y(\{ i\})\,,
 $$
 because $\bolda\, [\,T,(i|\emptyset )\,]=0$ for $|T|\geq 2$.
 This gives (\ref{eq.a1}). If $i\in N$, $B\subseteq N\setminus\{ i\}$ with
 $|B|=1$, then $\bolda\, [\,T,(i|B)\,]=\delta_{\{i\}\cup B}(T)$ for $|T|\geq 2$
 and one can write
 \begin{eqnarray*}
 \lefteqn{\sum_{\emptyset\neq T\subseteq N}
 \bolda\, [\,T,(i|B)\,]\cdot y(T)}\\
 &=&
 \sum_{|T|=1} \delta_{\{i\}}(T)\cdot y(T)
 + \sum_{|T|\geq 2} \delta_{\{i\}\cup B}(T)\cdot y(T)
 = y(\{ i\})+ y(\{ i\}\cup B)\,,
 \end{eqnarray*}
 which leads to (\ref{eq.a2}) for $S=\{ i\}\cup B$.
 Finally, if  $i\in N$, $B\subseteq N\setminus\{ i\}$ with $|B|\geq 2$
 then
 \begin{eqnarray*}
 \lefteqn{\sum_{\emptyset\neq T\subseteq N}
 \bolda\, [\,T,(i|B)\,]\cdot y(T)}\\
 &=&
 \sum_{|T|=1} \delta_{\{i\}}(T)\cdot y(T)
 + \sum_{|T|\geq 2} \{\delta_{\{i\}\cup B}(T)-\delta_{B}(T)\}\cdot
 y(T)\\
 &=& y(\{ i\})+ y(\{ i\}\cup B)-y(B)\,,
 \end{eqnarray*}
 which leads to (\ref{eq.a3}) for $S=\{ i\}\cup B$.
 }

 The next step is to show that $\{y\in {\dv R}^{m};
 \boldA^{\top}y\geq 0\}$ is a pointed (rational polyhedral) cone
 and characterize its extreme rays. In fact, we show that the rays correspond
 to non-empty classes of sets $\calA\subseteq {\cal P}_{1}(N)$
 closed under supersets. More specifically, we ascribe a vector
 $y_{\calA}\in {\dv R}^{m}$ to any such class $\calA$ by:
 \begin{equation}
 y_{\calA}(T) \equiv\delta (T\in\calA) -
 |\{j\in N;~ \{j\}\in\calA ~\&~ \{j\}\subset T\}|
 \qquad \mbox{for $T\in {\cal P}_{1}(N)$}\,.
 \label{eq.extreme-ray}
 \end{equation}
 Here is the crucial observation:

 \begin{lemma}\label{lem.4}
 A vector $y\in {\dv R}^{m}$ satisfies (\ref{eq.a1})-(\ref{eq.a3})
 if and only if it is a conic combination (= a linear combination with non-negative
 real coefficients) of vectors $y_{\calA}$ for classes
 $\emptyset\neq\calA\subseteq {\cal P}_{1}(N)$ closed under supersets.
 \end{lemma}

 \begin{proof}
 First, we leave to the reader to verify that any such vector
 $y_{\calA}$ satisfies (\ref{eq.a1})-(\ref{eq.a3}), which implies
 the sufficiency of the condition.

 To verify the converse implication, we ascribe to any
 $y\in {\dv R}^{m}$ satisfying (\ref{eq.a1})-(\ref{eq.a3}) the
 class of sets
 $$
 \calA_{y} =\{ S\in {\cal P}_{1}(N);\ \exists\, T\in {\cal P}_{1}(N),\
 T\subseteq S ~~ y(T)\neq 0\}\,,
 $$
 which is clearly closed under supersets and non-empty if $y\neq 0$.
 The idea is to prove the converse implication by induction on $|\calA_{y}|$.
 If $|\calA_{y}|=0$ then $y\equiv 0$ and the claim that $y$ is a conic
 combination of those vectors is evident. If $|\calA_{y}|\geq 1$ then it is enough
 to find some $\beta >0$ such that $y^{\prime}\equiv y-\beta\cdot y_{\calA}$
 satisfies (\ref{eq.a1})-(\ref{eq.a3}) and $|\calA_{y^{\prime}}|<|\calA_{y}|$.

 Since now we fix $y\in {\dv R}^{m}$, $y\neq 0$ satisfying
 (\ref{eq.a1})-(\ref{eq.a3}) and put:
 $$
 \calA\equiv\calA_{y}, \quad y_{*}\equiv y_{\calA},
 \quad Y\equiv\{ i\in N;\ y(\{ i\})\neq 0\}\,.
 $$
 Observe a few basic facts:
 $$
 y(S)=0 \quad \mbox{for $S\in {\cal P}_{1}(N)\setminus\calA$},
 \qquad y(S)>0 \quad \mbox{for $S\in {\calA}_{\min}$}\,.
 $$
 {\small\sl
 Indeed, assuming $S\in {\calA}_{\min}$ one has $y(S)\neq 0$.
 If $|S|=1$ then (\ref{eq.a1}) implies $y(S)>0$. If $|S|=2$
 then $\{ i\}\not\in\calA$ for both $i\in S$.
 Hence, $y(\{ i\})=0$ and (\ref{eq.a2}) gives $y(S)>0$.
 If $|S|\geq 3$ and $i\in S$, then both $\{ i\}\not\in\calA$
 and $S\setminus \{ i\}\not\in\calA$ and (\ref{eq.a3}) gives $y(S)>0$.
 }

 In particular, since $\{ j\}\in {\calA}_{\min}$ for $j\in Y$,
 and $\{i\}\not\in\calA$ for $i\not\in Y$,
 \begin{eqnarray}
 \beta\equiv\min\, \{y(T); \
 T\in{\calA}_{\min}\}>0,\,\mbox{and} &~~&
 y(\{j\})\geq\beta>0 ~~\mbox{for $j\in Y$}, \label{eq.a4}\\
 && y(\{i\})=0 ~~\mbox{for $i\in N\setminus Y$} \label{eq.4b}\,.
 \end{eqnarray}
 Further, we observe that $y$ is non-decreasing set function
 on subsets of $N\setminus Y$.
 {\small\sl
 Indeed, it is enough to show
 $T\subseteq S\subseteq N\setminus Y,\ |S\setminus T|=1 ~\Rightarrow ~ y(S)\geq y(T)$.
 Take $S\setminus T=\{ i\}$; then $i\not\in Y$ and $y(\{ i\})=0$.
 If $|S|=2$ then $T=\{ j\}$ with $j\not\in Y$ and
 $y(S)=y(S)+y(\{ i\})\geq 0= y(T)$
 follows from (\ref{eq.a2}) and (\ref{eq.4b}). If $|S|\geq 3$ then (\ref{eq.a3})
 says $y(S)+0-y(T)\geq 0$.
 }

 This implies:
 \begin{equation}
 S\in\calA\,,\ S\cap Y=\emptyset ~~\Rightarrow ~~ y(S)\geq\beta\,.
 \label{eq.a5}
 \end{equation}
 {\small\sl
 Indeed, it is enough to find $T\in\calA_{\min}$, $T\subseteq S$
 (of course, $T\cap Y=\emptyset$) and combine $y(S)\geq y(T)$ with
 $y(T)\geq \beta$, which follows from the definition of $\beta$
 in (\ref{eq.a4}).
 }

 Finally, also have:
 \begin{equation}
 S\in {\cal P}_{1}(N),\ |S\cap Y|\leq 1 ~~\Rightarrow ~~ y(S)\geq 0\,.
 \label{eq.a6}
 \end{equation}
 {\small\sl
 Indeed, this was verified in cases $|S|=1$ and $|S\cap Y|=0$ in
 (\ref{eq.a4})-(\ref{eq.a5}). Assume
 $|S|\geq 2$ and $|S\cap Y|=1$ and use the induction on $|S|$.
 If $|S|=2$ then $S=\{ i,j\}$ with
 $i\not\in Y$ and $j\in Y$ and (\ref{eq.a2})+(\ref{eq.4b}) give $y(S)\geq -y(\{
 i\})=0$. If $|S|\geq 3$ then choose $i\in S\setminus Y$ and write
 by (\ref{eq.a3})+(\ref{eq.4b}) $y(S)\geq y(S\setminus\{ i\})-y(\{i\})=y(S\setminus\{
 i\})$. Now, $y(S\setminus\{ i\})\geq 0$ follows from the
 induction premise.
 }

 To smooth later considerations let us gather the observations
 about $y_{*}=y_{\calA}$ defined in (\ref{eq.extreme-ray}).
 For singletons we have:
 $$
 y_{*}(\{i\})=1 ~~\mbox{ for $i\in Y$}, \qquad
 y_{*}(\{i\})=0 ~~\mbox{ for $i\not\in Y$}.
 $$
 Given $S\subseteq N$, $|S|=2$ we have:
 \begin{eqnarray*}
 y_{*}(S)=1 &~~& \mbox{if $S\cap Y=\emptyset$, $S\in\calA$.}\\
 y_{*}(S)=0 &~~& \mbox{if either [\,$S\cap Y=\emptyset\ \& \
 S\not\in\calA$\,] or $|S\cap Y|=1$.}\\
 y_{*}(S)=-1 &~~& \mbox{if $S\subseteq Y$.}
 \end{eqnarray*}
 For $S\subseteq N$, $|S|\geq 3$ we have:
 \begin{eqnarray*}
 y_{*}(S)=1 &~~& \mbox{if $S\cap Y=\emptyset$, $S\in\calA$.}\\
 y_{*}(S)=0 &~~& \mbox{if $S\cap Y=\emptyset$, $S\not\in\calA$.}\\
 y_{*}(S)=1-|S\cap Y| &~~& \mbox{if $S\cap Y\neq\emptyset$.}
 \end{eqnarray*}

 To show that
 $$
 y^{\prime}\equiv y-\beta\cdot y_{*}
 $$
 satisfies (\ref{eq.a1}), that is, $y^{\prime}(\{ i\})\geq 0$ for $i\in N$,
 we distinguish two cases.
 \begin{itemize}
 \item If $i\not\in Y$ then $y_{*}(\{ i\})=0$ and (\ref{eq.a1})
 for $y$ implies the same equality for $y^{\prime}$.
 \item If $i\in Y$ then $y^{\prime}(\{ i\})=y(\{ i\})-\beta\cdot y_{*}(\{ i\})
 = y(\{ i\})-\beta\cdot 1=y(\{ i\})-\beta\geq 0$ owing to
 (\ref{eq.a4}).
 \end{itemize}

 To show that $y^{\prime}$ satisfies (\ref{eq.a2}), that is,
 $y^{\prime}(S)+y^{\prime}(\{i\})\geq 0$ for $S\subseteq N$,
 $|S|=2$ and $i\in S$ we distinguish five cases.
 \begin{itemize}
 \item If $S\subseteq Y$ then $i\in Y$ and $y_{*}(S)+y_{*}(\{i\})=(-1)+1=0$
 and (\ref{eq.a2}) for $y$ implies the same equality for $y^{\prime},$ no matter
 what $\beta$ is.
 \item If $|S\cap Y|=1$, $i\not\in Y$ then $y_{*}(S)+y_{*}(\{i\})=0+0=0$
 and (\ref{eq.a2}) for $y$ implies what is desired, for the same reason.
 \item If $|S\cap Y|=1$, $i\in Y$ then $y^{\prime}(S)+y^{\prime}(\{i\})=
 y(S)-\beta\cdot y_{*}(S)+y(\{ i\})-\beta\cdot y_{*}(\{i\})=
 y(S)-\beta\cdot 0+y(\{ i\})-\beta\cdot 1=
 y(S)+y(\{ i\})-\beta$. However, $y(\{ i\})-\beta\geq 0$ by (\ref{eq.a4}) and
 $y(S)\geq 0$ by (\ref{eq.a6}), which implies what is desired.
 \item If $S\cap Y=\emptyset$, $S\not\in\calA$ then $i\not\in Y$ and
 $y_{*}(S)+y_{*}(\{i\})=0+0=0$
 and (\ref{eq.a2}) for $y$ implies what is desired,
 \item If $S\cap Y=\emptyset$, $S\in\calA$ then $i\not\in Y$ and
 by (\ref{eq.4b}) $y^{\prime}(S)+y^{\prime}(\{i\})=
 y(S)-\beta\cdot y_{*}(S)+y(\{ i\})-\beta\cdot y_{*}(\{i\})=
 y(S)-\beta\cdot 1+0-\beta\cdot 0= y(S)-\beta$. The desired inequality
 follows from (\ref{eq.a5}).
 \end{itemize}

 To show that $y^{\prime}$ satisfies (\ref{eq.a3}), that is,
 $y^{\prime}(S)+y^{\prime}(\{i\})-y^{\prime}(S\setminus\{ i\})\geq 0$ for
 $S\subseteq N$, $|S|\geq 3$ and $i\in S$ we distinguish seven cases.
 \begin{itemize}
 \item If $S\cap Y=\emptyset$, $S\setminus\{ i\}\in\calA$ then $S\in\calA$
 and $i\not\in Y$. Thus, $y_{*}(S)+y_{*}(\{i\})-y_{*}(S\setminus\{ i\})=(+1)+0-(+1)=0$
 and (\ref{eq.a3}) for $y$ implies the same inequality for $y^{\prime}$.
 \item If $S\cap Y=\emptyset$, $S\not\in\calA$ (which implies
 $S\setminus\{ i\}\not\in\calA$) then
 $y_{*}(S)+y_{*}(\{i\})-y_{*}(S\setminus\{ i\})=0+0-0=0$
 and (\ref{eq.a3}) for $y$ implies what is desired.
 \item If $S\cap Y=\emptyset$, $S\in\calA$, $S\setminus\{ i\}\not\in\calA$
 then $y(\{i\})=0=y(S\setminus\{ i\})$ and we can write
 $y^{\prime}(S)+y^{\prime}(\{i\})-y^{\prime}(S\setminus\{ i\})=
 y(S)-\beta\cdot y_{*}(S)+y(\{ i\})-\beta\cdot y_{*}(\{i\})-y(S\setminus\{ i\})
 +\beta\cdot y_{*}(S\setminus\{ i\})=
 y(S)-\beta\cdot 1+0-\beta\cdot 0-0 +\beta\cdot 0=
 y(S)-\beta$, which is non-negative by (\ref{eq.a5}).
 \item If $S\cap Y\neq\emptyset$, $i\not\in Y$ then
 $S\cap Y=(S\setminus\{ i\})\cap Y$ and
 $y_{*}(S)+y_{*}(\{i\})-y_{*}(S\setminus\{ i\})=
 (+1-|S\cap Y|)+0-(+1-|(S\setminus\{ i\})\cap Y|)=0$.
 Thus, (\ref{eq.a3}) for $y$ implies what is desired.
 \item If $S\cap Y\neq\emptyset$, $i\in Y$, $(S\setminus\{ i\})\cap Y\neq\emptyset$
 then $|S\cap Y|=1+|(S\setminus\{ i\})\cap Y|$ and
 $y_{*}(S)+y_{*}(\{i\})-y_{*}(S\setminus\{ i\})=
 (+1-|S\cap Y|)+(+1)-(+1-|(S\setminus\{ i\})\cap Y|)=0$.
 Thus, (\ref{eq.a3}) for $y$ implies what is desired.
 \item If $S\cap Y\neq\emptyset$, $i\in Y$, $(S\setminus\{ i\})\cap Y=\emptyset$,
 $S\setminus\{ i\}\in\calA$
 then $|S\cap Y|=1$ and
 $y_{*}(S)+y_{*}(\{i\})-y_{*}(S\setminus\{ i\})=
 (+1-1)+(+1)-(+1)=0$ and (\ref{eq.a3}) for $y$ implies what is desired.
 \item If $S\cap Y\neq\emptyset$, $i\in Y$, $(S\setminus\{ i\})\cap Y=\emptyset$,
 $S\setminus\{ i\}\not\in\calA$ then also $|S\cap Y|=1$ and
 $y(S\setminus\{ i\})=0$, which allows us to write
 $y^{\prime}(S)+y^{\prime}(\{i\})-y^{\prime}(S\setminus\{ i\})=
 y(S)-\beta\cdot y_{*}(S)+y(\{ i\})-\beta\cdot y_{*}(\{i\})-y(S\setminus\{ i\})
 +\beta\cdot y_{*}(S\setminus\{ i\})=
 y(S)-\beta\cdot (1-1)+y(\{ i\})-\beta\cdot 1-0 +\beta\cdot 0=
 y(S)+y(\{ i\})-\beta$. However, $y(\{ i\})-\beta\geq 0$ by (\ref{eq.a4}) and
 $y(S)\geq 0$ by (\ref{eq.a6}), which implies what is desired.
 \end{itemize}

 Thus, $y^{\prime}$ satisfies (\ref{eq.a1})-(\ref{eq.a3}) and,
 because of the choice of $\beta$, $y^{\prime}(T)=0$ for at least
 one $T\in\calA_{\min}$ and $|\calA_{y^{\prime}}|<|\calA_{y}|$,
 which concludes the induction step. Indeed, realize that, by (\ref{eq.extreme-ray}),
 $y_{*}(T)\equiv y_{\calA}(T)=0$ for $T\in {\cal P}_{1}(N)\setminus\calA$.
 \end{proof}

 Now, Lemma \ref{lem.4} allows us to re-formulate the requirement
 (\ref{eq.problem-3}) in the form of finitely many conditions on
 $\u$:
 \begin{equation}
 \forall\, \emptyset\neq\calA\subseteq {\cal P}_{1}(N)
 ~\mbox{closed under supersets} \quad
 \boldb_{\usc}^{\top}y_{\calA}\geq 0\,.
 \label{eq.problem-4}
 \end{equation}
 {\small\sl
 Indeed, if $y\in {\dv R}^{m}$ is such that $\boldA^{\top}y\geq
 0$ and $y=\sum \lambda_{\calA}\cdot y_{\calA}$,
 $\lambda_{\calA}\geq 0$, then
 $\boldb_{\usc}^{\top}y=
 \sum \lambda_{\calA}\cdot\boldb_{\usc}^{\top} y_{\calA}\geq 0$. }

 It remains to reformulate, given such an $\calA$, the condition
 $\boldb_{\usc}^{\top}y_{\calA}\geq 0$. Assuming $|N|\geq 2$, denote
 for this purpose $A\equiv\{i\in N;\ \{i\}\in\calA\}$ and write using the
 definition of $\boldb_{\usc}$ and $y_{\calA}$ from
 (\ref{eq.extreme-ray}):
 \small
 \begin{eqnarray*}
 0 &\leq& \boldb_{\usc}^{\top}y_{\calA} =
 \sum_{|T|\geq 1} \bepa_{\usc}(T)\cdot y_{\calA}(T) =
% &=&
 \sum_{|T|=1} y_{\calA}(T) + \sum_{|T|\geq 2}
 \{\delta_{N}(T)-\us(T)\}\cdot y_{\calA}(T)\\
 &=&
 \sum_{|T|=1} y_{\calA}(T) +  y_{\calA}(N) -
 \sum_{|T|\geq 2}\us(T)\cdot y_{\calA}(T)\\
 &=& |A| + \underbrace{(1-|A|)}_{y_{\calA}(N)} -
 \sum_{|T|\geq 2}\us(T)\cdot y_{\calA}(T) = 1 -
 \sum_{|T|\geq 2}\us(T)\cdot y_{\calA}(T)\,.
 \end{eqnarray*}
 \normalsize
 Thus, $\boldb_{\usc}^{\top}y_{\calA}\geq 0$ is equivalent to
 $\sum_{|T|\geq 2}\u(T)\cdot y_{\calA}(T)\leq 1$. To get even
 more elegant form of it, assume $\u$ satisfies (\ref{eq.standardize})
 and observe
 \begin{eqnarray*}
 \lefteqn{\sum_{|T|\geq 2} \u(T)\cdot |T\cap A| =
 \sum_{|T|\geq 2} \u(T)\cdot \sum_{i\in A} \delta(i\in T) =
 \sum_{|T|\geq 2}\, \sum_{i\in A}\, \u(T)\cdot \delta(i\in T)}\\
 &=& \sum_{i\in A}\,  \sum_{|T|\geq 2} \u(T)\cdot \delta(i\in T)=
 \sum_{i\in A} \, \sum_{|T|\geq 2,\,i\in T} \u(T)
 \stackrel{(\ref{eq.standardize})}{=} \sum_{i\in A} -\u(\{i\})\,.
 \end{eqnarray*}
 Therefore, we can write by (\ref{eq.extreme-ray}):
 \begin{eqnarray*}
 \lefteqn{\sum_{|T|\geq 2} \u(T)\cdot y_{\calA}(T) =
 \sum_{|T|\geq 2} \u(T)\cdot \delta(T\in\calA) -
 \sum_{|T|\geq 2} \u(T)\cdot |T\cap A|}\\
 &=&\sum_{|T|\geq 2} \u(T)\cdot \delta(T\in\calA) + \sum_{i\in A}\,
 \u(\{i\})= \sum_{|T|\geq 1} \u(T)\cdot \delta(T\in\calA)
 = \sum_{T\in\calA} \u(T)\,,
 \end{eqnarray*}
 which means that $\boldb_{\usc}^{\top}y_{\calA}\geq 0$ is
 equivalent to $\sum_{T\in\calA} \u(T)\leq 1$. Thus, under
 validity of (\ref{eq.standardize}), (\ref{eq.problem-4})
 is equivalent to (\ref{eq.specific-con}) and we have:

 \begin{corollary}\label{cor.nonneg-to-speci}
 Provided $|N|\geq 2$, the condition (\ref{eq.problem-1}) for
 $\u\in {\dv R}^{{\calP}}$ is equivalent to the simultaneous validity of
 (\ref{eq.standardize}) and (\ref{eq.specific-con}).
 \end{corollary}

 \subsubsection{Remarks on the matrix \boldA }\label{sssec.A-hermite}

 Consider again the $m\times n$ matrix $\boldA$ defined in (\ref{eq.def-A-matrix});
 recall that $m=2^{|N|}-1$ and $n=|N|\cdot 2^{|N| -1}$. We have observed
 in \S\,\ref{sssec.trans-elem-stand} that $\boldA$ plays a central role in
 the transition from $\boldeta$-vectors to standard imsets. Now, we show
 that $\boldA$ has full row rank by deriving its its Hermite normal form
 (see \S\,4.1 in \cite{Sch86} for this concept).

 \begin{proposition}\label{prop.A-matrix}
 The matrix $\boldA$ has Hermite normal form $[\boldI \; \boldzero ]$,
 where $\boldI$ is the $m\times m$ identity matrix and $\boldzero$
 the $m\times (n-m)$ zero matrix.
 \end{proposition}

 \begin{proof}
 The columns of $\boldA$ are indexed by pairs $(i|B)$ and given by
 $$
 \begin{array}{ll}
 \boldA_{(i|\emptyset)} = \delta_{\{i\}} \qquad\qquad &
  \mbox{for $i\in N$},\\
 \boldA_{(i|j)} = \delta_{\{i\}} + \delta_{\{i,j\}}  &
 \mbox{for $i,j\in N$, $i\neq j$}\\
 \boldA_{(i|B)} = \delta_{\{i\}} - \delta_{B} + \delta_{\{i\}\cup B}
 & \mbox{for $i\in N$, $B\subseteq N\setminus\{ i\}$, $|B|\geq 2$}.
 \end{array}
 $$
 Thus, $\delta_{\{i\}}=\boldA_{(i|\emptyset)}$ and $\delta_{\{i,j\}}=\boldA_{(i|j)} -
 \boldA_{(i|\emptyset)}$. To show by induction on $|T|\geq 1$ that
 $\delta_{T}$ can be written as an integer combination
 of columns of $\boldA$, assume $|T|\geq 3$ and choose a pair
 $(i|B)$ with $T=\{ i\}\cup B$, $|B|\geq 2$. Then
 $$
 \delta_{T} = \delta_{\{i\}\cup B}= \boldA_{(i|B)} - \delta_{\{i\}}
 + \delta_{B},
 $$
 where, by the induction hypothesis, the terms $\delta_{\{i\}}$ and $\delta_{B}$ can be
 written as integer combination of the columns of \boldA.

 Thus, using elementary columns operations, $\boldA$ can be transformed such that it
 contains all $m$ elementary column vectors $\delta_{T}$. Using additional
 column operations, all other columns can be zeroed out. Therefore, using
 elementary column operations, $\boldA$ can be transformed to the form
 $[\boldI \; \boldzero ]$.
 \end{proof}

 Before writing this report, we verified computationally that $\boldA$ is unimodular,
 strongly unimodular, strongly k-modular, however not totally unimodular
 for $ 3 \leq |N|\leq 6$ using software written by Matthias Walther available at
 \url{https://github.com/xammy/unimodularity-test}.
% $$
% {\tt https://github.com/xammy/unimodularity-test}.
% $$
 This led us to a hypothesis that $\boldA$ is unimodular for any
 $|N|$. In \S\,\ref{sssec.remark-char}, we confirm this hypothesis.

 \subsubsection{Translation to the framework of characteristic
 imsets}\label{sssec.trans-elem-char}

 We observed in \S\,\ref{sssec.trans-elem-stand} that Jaakkola et al.'s
 elementary constraints (\ref{eq.jaa-non-neg})-(\ref{eq.jaa-equal}) are
 transformed into $\u$-constraints as (\ref{eq.standardize})-(\ref{eq.specific-con}).
 Transforming (\ref{eq.standardize})-(\ref{eq.specific-con}) into
 $\c$-constraints is a simpler task because of the one-to-one
 correspondence $\u\leftrightarrow\c$ (see \S\,\ref{sec:charimset}).
 We already know that (\ref{eq.standardize}) takes the form of tacit
 restrictions on $\c$-vectors (\ref{eq.equal-char}).
 As concerns the specific inequality constraints (\ref{eq.specific-con}),
 we show below that every such inequality, for
 $\emptyset\neq\calA\subseteq {\cal P}_{1}(N)$ closed under
 supersets, is transformed into the framework of $\c$-vectors as follows:
 \begin{equation}
 0\leq \sum_{S\subseteq N} \kappa_{\calA}(S)\cdot\c (S)\,,
 \label{eq.spec-char}
 \end{equation}
 where the coefficients $\kappa_{\calA}(-)$ are given by
 \begin{equation}
 \kappa_{\calA}(S)\equiv\sum_{T\in \calA,\, T\subseteq S}
 (-1)^{|S\setminus T|}\qquad \mbox{for $S\subseteq N$}\,.
 \label{eq.kappa-def}
 \end{equation}
 However, the formula (\ref{eq.kappa-def}) is not suitable to
 compute the coefficients. It is more appropriate to introduce
 them equivalently in terms of the class $\calA_{\min}\equiv\calI$
 of minimal sets in $\calA$. More specifically, let us introduce the class
 $\calC (\calI )$ of possible unions of sets from a non-empty class
 $\calI\subseteq {\cal P}_{1}(N)$ of incomparable sets:
 $$
 \calC (\calI )\equiv\{ S\subseteq N;\ \exists\,
 \emptyset\neq\calK\subseteq\calI \quad
 \mbox{such that $S=\bigcup_{T\in\calK} T$}\}.
 $$
 Then can can compute the coefficients $\kappa_{\calA}(-)$
 recursively as follows:
 \begin{equation}
 \begin{array}{lcl}
 \kappa_{\calA}(S)=0 && \mbox{if $S\subseteq N$,
 $S\not\in\calC (\calI )$},\\[0.3ex]
 \kappa_{\calA}(S)= 1-\sum\limits_{T\in\calC (\calI ),\, T\subset S}
 \kappa_{\calA}(T) && \mbox{for $S\in\calC (\calI )$}\,.
 \end{array}
 \label{eq.kappa-recur}
 \end{equation}
 This implies that $\kappa_{\calA}(S)=1$ for $S\in\calA_{\min}=\calI$
 and that $\kappa_{\calA}$ has the more zeros the smaller
 $|\calA_{\min}|$ is. Therefore, in the framework of characteristic imsets,
 it is more convenient to ascribe the (transformed) specific inequality
 constraints directly to classes $\emptyset\neq\calI\subseteq {\cal P}_{1}(N)$
 of incomparable sets.

 \begin{lemma}\label{lem.5}
 Let $\u$ and $\c$ be imsets related by (\ref{eq.portrait})-(\ref{eq.characteristic})
 and $\emptyset\neq\calA\subseteq{\cal P}_{1}(N)$ a class of sets closed under
 supersets. Then the inequality (\ref{eq.specific-con}) corresponding to
 $\calA$ has the form (\ref{eq.spec-char}), where the coeficients
 $\kappa_{\calA}(-)$ are given by (\ref{eq.kappa-recur}).
 \end{lemma}

 \begin{proof}
 The first observation is that the coefficients given by
 (\ref{eq.kappa-def}) satisfy
 \begin{equation}
 \kappa_{\calA}(S)=0 ~~\mbox{for $S\subseteq N$, $S\not\in\calA$}, \quad
 \mbox{and}\quad \sum_{S\subseteq N} \kappa_\calA(S)=\sum_{S\in \calA} \kappa_\calA(S)=1\,.
 \label{eq.sum-kappa}
 \end{equation}
 To verify it realize that $\calA\subseteq\calP$ is
 closed under supersets and write:
 \begin{eqnarray*}
 \sum_{S\in \calA} \kappa_{\calA}(S) &=& \sum_{S\in \calA} \
 \sum_{T\in \calA,\, T\subseteq S} (-1)^{|S\setminus T|}
 = \sum_{T\in \calA} \ \sum_{S\in \calA,\, T\subseteq S}  (-1)^{|S\setminus T|}\\
 &=& \sum_{T\in \calA} \ \sum_{S,\, T\subseteq S\subseteq N}  (-1)^{|S\setminus T|}
 = \sum_{T\in \calA} \delta_{N}(T) =1\,.
 \end{eqnarray*}
 To see that (\ref{eq.specific-con}) is transformed into (\ref{eq.spec-char})
 we substitute the inverse formula (\ref{eq.inverse}) into it and use
 the fact $\calA$ is closed under supersets:
 \begin{eqnarray*}
 1 &\geq & \sum_{T\in\calA} \u (T)
 = \sum_{T\in\calA} \ \sum_{S,\, T\subseteq S\subseteq N}
 (-1)^{|S\setminus  T|}\cdot \p (S)\\
 &=&  \sum_{S\in\calA} \ \sum_{T\in \calA,\, T\subseteq S}\p (S)\cdot (-1)^{|S\setminus
 T|}
 =  \sum_{S\in \calA} \p (S)\cdot
 \underbrace{\sum_{T\in \calA,\, T\subseteq S}(-1)^{|S\setminus
 T|}}_{\kappa_\calA(S)}\,.
 \end{eqnarray*}
 Thus, substitute (\ref{eq.sum-kappa}) in that inequality and
 get
 $$
 0\leq 1-\sum_{S\in\calA} \p (S)\cdot\kappa_{\calA}(S) \stackrel{(\ref{eq.sum-kappa})}{=}
 \sum_{S\in\calA} \kappa_{\calA}(S)
 -\sum_{S\in\calA} \kappa_{\calA}(S)\cdot\p (S) =
 \sum_{S\in\calA} \kappa_{\calA}(S)\cdot
 \underbrace{[1- \p (S)]}_{\c (S)}\,,
 $$
 which is, owing to (\ref{eq.characteristic}) and (\ref{eq.sum-kappa}),
 nothing but (\ref{eq.spec-char}).

 It remains to show that (\ref{eq.kappa-def}) takes the form
 (\ref{eq.kappa-recur}). An auxiliary fact is
 \begin{equation}
 \forall\, S\in {\cal A} \qquad
 \sum_{T\subseteq S} \kappa_{\calA} (T)=1\,.
 \label{eq.kappa}
 \end{equation}
 Indeed, to see it, consider the class
 ${\cal A}_{S}\equiv\{ T\subseteq S;\ T\in {\cal A}\}$, which is a
 class of subsets of $S$, closed under supersets. Moreoever,
 for any $T\in {\cal A}_{S}$, one has $\kappa_{\cal A}(T)=\kappa_{{\cal A}_{S}}(T)$,
 which implies by (\ref{eq.sum-kappa}) applied to ${\cal A}_{S}$
 and $S$ in place of $N$ that
 $$
 1= \sum_{T\in {\calA}_{S}} \kappa_{{\cal A}_{S}}(T)=
 \sum_{T\in {\cal A},\, T\subseteq S} \kappa_{\cal A}(T) =
 \sum_{T\subseteq S} \kappa_{\calA} (T)\,.
 $$
 In the rest of the proof we write $\calI$ in place of ${\calA}_{\min}$
 and omit the index in $\kappa_{\calA}(-)$ and write $\kappa (-)$ only.
 For every $\calK\subseteq\calI$ we introduce the class of sets
 whose only subsets in $\calI$ are elements of $\calK$:
 $$
 \calB_{\calK}\equiv \{ S\subseteq N;\ K\subseteq S ~\mbox{for
 $K\in\calK$} ~~ \& ~~ L\setminus S\neq\emptyset
 ~\mbox{for $L\in\calI\setminus\calK$}
 ~\}.
 $$
 Of course, it may happen that $\calB_{\calK}$ is empty for some
 $\calK\subseteq\calI$. Nevertheless, the collection of
 classes $\calB_{\calK}$, where $\calK$ runs over subsets of $\calI$,
 form a partition of $\calP$.  Moreover, every non-empty class
 $\calB_{\calK}$ has the least set (in sense of inclusion),
 namely $S_{\calK}\equiv\bigcup_{T\in\calK} T$.
 Observe that $\calK=\emptyset$ leads to a non-empty class
 $\calB_{\emptyset}=\calP\setminus\calA$ with $S_{\emptyset}=\emptyset$.
 Since $\calI$ consists of incomparable sets, every $S\in\calI$ belongs to
 just one $\calB_{\calK}$ with $|\calK |=1$, namely $\calK=\{ S\}$.
 The class $\calC (\calI )$ defined above (\ref{eq.kappa-recur})
 then coincides with $\{ S_{\calK};\
 \emptyset\neq\calK\subseteq\calI  ~~\mbox{with}
 ~\calB_{\calK}\neq\emptyset\,\}$.

 An easy consequence of (\ref{eq.kappa-def}) is that $\kappa (S)=0$ for
 $S\in\calB_{\emptyset}$ (= $S\not\in\calA$) and $\kappa(S)=1$ for $S\in\calI$.
 To verify (\ref{eq.kappa-recur}) it is enough to show by induction
 on $|\calK|$ the following two statements:
 \begin{description}
 \item[(i)] $\kappa(S)=0$ for $S\in\calB_{\calK}$, $S\neq
 S_{\calK}$,
 \item[(ii)] $\forall\, |\calK|\geq 1 ~\mbox{with $\calB_{\calK}\neq\emptyset$} \quad
 1= \sum_{\calL\subseteq\calK ,\, \calB_{\calL}\neq\emptyset}\, \kappa(S_{\calL})$.
 \end{description}
 Indeed, this is because for $\calL ,\calK\subseteq\calI$ with
 $\calB_{\calL}\neq\emptyset\neq\calB_{\calK}$ one has
 $\calL\subseteq\calK$ if and only if $S_{\calL}\subseteq S_{\calK}$.
 We already know this is true in case $|\calK|=0$. Now assume
 $|\calK|\geq 1$ and the statements hold for any
 $\calL\subset\calK$. Consider arbitrary $S\in\calB_{\calK}$ and
 write using (\ref{eq.kappa}) and the fact that subsets
 of $S$ must belong to $\calB_{\calL}$ for $\calL\subseteq\calK$:
 \begin{eqnarray}
 1 &\stackrel{(\ref{eq.kappa})}{=}& \sum_{T\subseteq S} \kappa(T) =
 \sum_{\calL\subseteq\calK ,\, \calB_{\calL}\neq\emptyset} \,
 \sum_{T\subseteq S,\, T\in\calB_{\calL}} \kappa(T) \nonumber\\
 &=& \sum_{T\subseteq S,\, T\in\calB_{\calK}} \kappa(T) +
 \sum_{\calL\subset\calK ,\, \calB_{\calL}\neq\emptyset} \,
 \sum_{T\subseteq S,\, T\in\calB_{\calL}} \kappa(T)\,.
 \label{eq.kap-1}
 \end{eqnarray}
 Now, observe that the induction premise (i) applied to
 any $\calL\subset\calK$, $\calB_{\calL}\neq\emptyset$ says that
 $\kappa$ vanishes in $\calB_{\calL}$ except for $S_{\calL}$. In particular, for any
 ${\cal D}\subseteq\calB_{\calL}$ with $S_{\calL}\in {\cal D}$
 one has $\sum_{T\in {\cal D}} \kappa(T)=\kappa (S_{\calL})$.
 This implies that the second term in (\ref{eq.kap-1}) is
 $\sum_{\calL\subset\calK ,\, \calB_{\calL}\neq\emptyset}
 \kappa (S_{\calL})$ and we have observed that
 \begin{equation}
 \forall\, S\in\calB_{\calK} \qquad
 \sum_{T\subseteq S,\, T\in\calB_{\calK}} \kappa(T)
 = 1-\sum_{\calL\subset\calK ,\, \calB_{\calL}\neq\emptyset} \kappa (S_{\calL}),
 \label{eq.kap-2}
 \end{equation}
 which means the function $S\mapsto \sum_{T\subseteq S,\, T\in\calB_{\calK}}
 \kappa(T)$ is constant on $\calB_{\calK}$. This allows one to
 derive (i) for $\calK$, for instance, by induction on $|S|$ for
 $S\in\calB_{\calK}$. If we apply (\ref{eq.kap-2}) to $S=S_{\calK}$ we get
 (ii) for $\calK$.
 \end{proof}

 \begin{example}\label{exa.6}\rm
 Take $N=\{a,b,c\}$ and classify types of considered classes $\calA$,
 specified by $\calA_{\min}$. Using (\ref{eq.kappa-recur}) we get the
 corresponding inequalities (\ref{eq.spec-char}):
 \begin{itemize}
 \item ${\cal A}_{\min}=\{ abc\}$ leads to $\kappa_{\cal A}(abc)=1$
 and $\kappa_{\cal A}(S)=0$ otherwise. This gives the constraint
 $0\leq \c(abc)$,
 \item ${\cal A}_{\min}=\{ ab\}$ leads to $\kappa_{\cal A}(ab)=1$
 (and $\kappa_{\cal A}(S)=0$ otherwise), which gives the constraint
 $0\leq \c(ab)$,
 \item ${\cal A}_{\min}=\{ ab,ac\}$ leads to $\kappa_{\cal A}(ab)=\kappa_{\cal A}(ac)=1$
 and $\kappa_{\cal A}(abc)=-1$, which gives the constraint
 $0\leq \c(ab)+\c(ac)-\c(abc)$,
 \item ${\cal A}_{\min}=\{ ab,ac,bc\}$ leads to
 $\kappa_{\cal A}(ab)=\kappa_{\cal A}(ac)=\kappa_{\cal A}(bc)=1$
 and $\kappa_{\cal A}(abc)=-2$, which gives the constraint
 $0\leq \c(ab)+\c(ac)+\c(bc)-2\c(abc)$,
 \item ${\cal A}_{\min}=\{ c\}$ leads to $\kappa_{\cal A}(c)=1$
 which gives $0\leq \c(c)$, which is a vacuous constraint because
 of $\c(c)=1$ implied by (\ref{eq.equal-char}),
 \item ${\cal A}_{\min}=\{ c,ab\}$ leads to $\kappa_{\cal A}(c)=\kappa_{\cal A}(ab)=1$
 and $\kappa_{\cal A}(abc)=-1$, and then to $0\leq \c(c)+\c(ab)-\c(abc)$, which leads
 after the substitution $\c(c)=1$ to $0\leq 1+\c(ab)-\c(abc)$,
 \item ${\cal A}_{\min}=\{ a,b\}$ leads to $\kappa_{\cal A}(a)=\kappa_{\cal A}(b)=1$
 and $\kappa_{\cal A}(ab)=-1$, and then, after substituing $\c(i)=1$, to $0\leq 2-\c(ab)$,
 \item ${\cal A}_{\min}=\{ a,b,c\}$ leads to
 $\kappa_{\cal A}(a)=\kappa_{\cal A}(b)=\kappa_{\calA}(c)=1$,
 $\kappa_{\cal A}(ab)=\kappa_{\cal A}(ac)=\kappa_{\calA}(bc)=-1$ and
 $\kappa_{\cal A}(abc)=1$, which gives, after the substitution $\c(i)=1$,
 $0\leq 3-\c(ab)-\c(ac)-\c(bc)+\c(abc)$.
 \end{itemize}
 Thus, we see that the non-vacuous constraints are identical
 with the transformed elementary $\boldeta$-constraints -- see Example \ref{exa.5}.
 \end{example}

 \subsubsection{Remarks on the characteristic
 transformation}\label{sssec.remark-char}

 Let us consider the characteristic transformation given by
 (\ref{eq.eta-to-char}) -- see \S\,\ref{sssec.eta-to-char}. It can be viewed as a mapping
 $\boldeta\mapsto\boldB\boldeta$, where $\boldB$ is an $m\times n$
 matrix, whose entries $\boldb\, [\,S,(i|B)\,]$ are specifed
 as follows: for $|S|\geq 1$, $i\in N$, $B\subseteq N\setminus\{ i\}$,
 \begin{equation}
 \boldb\, [\,S,(i|B)\,] =
 \delta (\,i\in S \ \& \ S\setminus\{ i\}\subseteq B\,)
 \equiv \delta (\, S\subseteq\{ i\}\cup B\,)- \delta (\, S\subseteq B\,)\,.
 \label{eq.def-B-matrix}
 \end{equation}
 There is a close relation to the matrix $\boldA$ introduced in
 (\ref{eq.def-A-matrix}). Indeed, there exists an invertible unimodular
 $m\times m$ matrix $\boldC$ such that $\boldB =\boldC\boldA$.
 More specifically, the entries $\boldc\, [\,S,T\,]$ of $\boldC$ for
 non-empty sets $S,T\subseteq N$ are given by
 $$
 \boldc\, [\,S,T\,] =\left\{
 \begin{array}{ll}
 \delta (\,S\subseteq T\,) & \mbox{if $|S|\geq 2$},\\
 \delta (\,S=T\,) & \mbox{if $|S|=1$}.
 \end{array}
 \right.
 $$
 To see it write for fixed $S\subseteq N$, $|S|\geq 2$ and a pair $(i|B)$ with help
 of (\ref{eq.def-A-matrix}):
 \begin{eqnarray*}
 \sum_{T\neq\emptyset} \boldc\, [\,S,T\,]\cdot\bolda\, [\,T,(i|B)\,]
 &=&\sum_{T\supseteq S} \bolda\, [\,T,(i|B)\,]
 = \sum_{T\supseteq S} [\,\delta_{\{i\}\cup B}(T)-\delta_{B}(T)\,]\\
&=& \sum_{T\supseteq S} \delta_{\{i\}\cup B}(T)
 - \sum_{T\supseteq S} \delta_{B}(T)\\
  &=& \delta(S\subseteq \{ i\}\cup B)  - \delta(S\subseteq B)
   =
  \boldb\, [\,S,(i|B)\,]\,.
 \end{eqnarray*}
 Analogously, for $S\subseteq N$, $|S|=1$ one has
 \small
 $$
 \sum_{T\neq\emptyset}  \boldc\,[\,S,T\,]\cdot\bolda\, [\,T,(i|B)\,] =
 \sum_{T=S} \bolda\, [\,T,(i|B)\,] = \bolda\, [\,S,(i|B)\,]
 =\delta_{\{ i\}}(S)=\boldb\, [\,S,(i|B)\,]\,.
 $$
 \normalsize
 We leave to the reader to verify that the $m\times m$-matrix
 $\boldD$ with entries $\boldd\, [\,T,R\,]$ for non-empty $T,R\subseteq N$ given by
 $$
 \boldd\, [\,T,R\,] =\left\{
 \begin{array}{ll}
 \delta (\,T\subseteq R\,)\cdot (-1)^{|R\setminus T|} & \mbox{if $|T|\geq 2$},\\
 \delta (\,T=R\,) & \mbox{if $|T|=1$}.
 \end{array}
 \right.
 $$
 is an inverse matrix to $\boldC$. Since both $\boldC$ and its
 inverse $\boldD$ are integral matrices, they are both unimodular.
 The following observation appears to be important.

 \begin{lemma}\label{lem.6}
 Both the matrix $\boldA$ given by (\ref{eq.def-A-matrix}) and the matrix $\boldB$
 given by (\ref{eq.def-B-matrix}) are full row rank unimodular matrices.
 \end{lemma}

 \begin{proof}
 Since $\boldA =\boldD\boldB$ where $\boldD$ is an invertible unimodular
 $m\times m$-matrix, it is enough to show that $\boldB$ is unimodular.
 By Proposition \ref{prop.A-matrix} and $\boldB =\boldC\boldA$ we
 already know that $\boldB$ has full row rank.

 To show it is unimodular we re-label its columns and add some new ones.
 The original columns of $\boldB$ corresponding to pairs $(i|B)$ with
 $B\neq\emptyset$ are re-labelled by pairs $(C:B)$ of sets $\emptyset\neq
 B\subseteq C\subseteq N$ with $|C\setminus B|=1$; that is, $(i|B)$
 is replaced by $(C:B)$ where $C=\{i\}\cup B$. The formula
 (\ref{eq.def-B-matrix}) implies
 $$
 \boldb\, [\,S,(C:B)\,] =
 \delta (\, S\subseteq C\,)- \delta (\, S\subseteq B\,)
 \qquad  \mbox{for $S\subseteq N$, $|S|\geq 1$}.
 $$
 The original column corresponding to a pair $(i|\emptyset)$, $i\in N$
 is re-labelled by a singleton set $R=\{ i\}$. Note that the column
 has the form $\delta_{R}$. The newly added columns are labelled
 by sets $R\subseteq N$, $|R|\geq 2$ and defined as follows:
 $$
 \boldb\, [\,S,R\,] = \delta (\, S\subseteq R\,)
 \qquad  \mbox{for $S\subseteq N$, $|S|\geq 1$}.
 $$
 Observe that this formula also holds in case $|R|=1$. Now, it is
 enough to show that the extended matrix $\boldB$ is unimodular.

 Let $\bar{\boldB}$ denote the $m\times m$-submatrix of $\boldB$
 corresponding to columns labelled by sets $\emptyset\neq R\subseteq N$.
 It follows from the above description of columns in $\boldB$ that
 $\boldB =\bar{\boldB}\boldE$ where the matrix $\boldE$ has the
 entries $\bolde\,[\,T,R\,]$ for $\emptyset\neq T,R\subseteq N$
 and $\bolde\,[\,T,(C:B)\,]$ for $\emptyset\neq T\subseteq N$,
 $\emptyset\neq B\subseteq C\subseteq N$, $|C\setminus B|=1$
 specified as follows:
 $$
 \begin{array}{lcl}
 \bolde\, [\,T,R\,] &=& \delta (\, T=R\,)\,, \\
 \bolde\, [\,T,(C:B)\,] &=& \delta (\, T=C\,)- \delta (\,
 T=B\,)\,.
 \end{array}
 $$
 Therefore, it is enough to show that $\bar{\boldB}$ is invertible
 unimodular matrix and $\boldE$ totally unimodular (cf. Theorem
 21.6 in \cite{Sch86}). We leave to the reader to verify that the
 inverse matrix $\boldF$ to $\bar{\boldB}$ has the entries
 $$
 \boldf\, [\,R,U\,] =\delta (\, R\subseteq U\,)\cdot (-1)^{|U\setminus R|}
 \qquad \mbox{for  $\emptyset\neq R,U\subseteq N$}.
 $$
 Since $\bar{\boldB}$ has integral inverse $\boldF$, it is
 unimodular. The matrix $\boldE$ is totally unimodular because it is the
 restriction of a network matrix (cf.\ \S\,19.3 of \cite{Sch86}).
 More specifically, one can add one dummy row to $\boldE$,
 labelled by $S=\emptyset$: put $\bolde\,[\,\emptyset ,R\,]=-1$
 for $\emptyset\neq R\subseteq N$ and $\bolde\,[\,\emptyset ,(C:B)\,]=0$
 for any pair $(C:B)$. We obtain a matrix with entries in $\{ -1,0,+1\}$
 such that each of its columns contains exactly once $+1$ and
 exactly once $-1$. As mentioned in the statement (18) of \S\,19.3
 in \cite{Sch86}, such a matrix is totally unimodular. Of course,
 it remains totally unimodular if the row corresponding to $S=\emptyset$
 is again removed.
 \end{proof}

 \subsection{Transformation of cluster
 inequalities}\label{ssec.trans-cluster}

 Luckily, these inequalites transform nicely to the framework of
 imsets.

 \begin{lemma}\label{lem.7}
 Provided $\boldeta$ satisfies (\ref{eq.jaa-equal}),
 the cluster inequality (\ref{eq.cluster-1}) for $C\subseteq N$, $|C|\geq 2$
 can be re-written either in terms of $\u$-vectors as
 \begin{equation}
 \sum_{T\subseteq N,\, |C\cap T|\geq 2} \,\, \u(T)\cdot (|C\cap T|-1)\geq 0\,,
 \label{eq.cluster-2}
 \end{equation}
 or in terms of $\c$-vectors as
 \begin{equation}
 |C|-1- \sum_{S\subseteq C,\, |S|\geq 2} ~
 \c(S)\cdot (-1)^{|S|}\geq 0\,.
 \label{eq.cluster-3}
 \end{equation}
 \end{lemma}

 \begin{proof}
 By (\ref{eq.cluster-1}), it is enough to show that the following equalities hold
 \begin{eqnarray*}
 \underbrace{|C|- \sum_{S\subseteq C,\, |S|\geq 2} ~ \c(S)\cdot (-1)^{|S|}}_{\equiv (\ast )}
 &=& 1+\sum_{T\subseteq N,\, |C\cap T|\geq 2} ~ \u(T)\cdot (|C\cap T|-1)\\
  &=& \sum_{i\in C} ~\sum_{B\subseteq N\setminus\{ i\},\, B\cap C=\emptyset}
 \eta(i|B)\,.
 \end{eqnarray*}
 Let $(\ast)$ denote the first expression there and write by
 (\ref{eq.characteristic})-(\ref{eq.portrait}):
 \small
 \begin{eqnarray*}
 \lefteqn{(\ast ) =
 |C|- \sum_{S\subseteq C,\, |S|\geq 2} ~
 (-1)^{|S|}\cdot [\,1-\sum_{T\supseteq S} \u(T)\,]}\\
 &=& |C|- \underbrace{\sum_{S\subseteq C,\, |S|\geq 2} ~
 (-1)^{|S|}}_{|C|-1} + \sum_{S\subseteq C,\, |S|\geq 2}
 ~ (-1)^{|S|}\cdot \sum_{T\supseteq S} \u(T)\\
 &=& 1+\sum_{S\subseteq C,\, |S|\geq 2} ~\sum_{T\supseteq S}
 \u(T)\cdot ~ (-1)^{|S|}
 = 1+ \sum_{T,\, |C\cap T|\geq 2} \u(T)\cdot
 \underbrace{\sum_{S\subseteq C\cap T,\, |S|\geq 2} (-1)^{|S|}}_{|C\cap
 T|-1}\,.
 \end{eqnarray*}
  \normalsize
 This already proves the first equality. Now, we substitute (\ref{eq.jaa-to-stand})
 in the last expression (note $|T|\geq 2$ for $T$ here) and change the order of summation:
 \footnotesize
 \begin{eqnarray*}
 (\ast ) &=& 1+ \sum_{T,\, |C\cap T|\geq 2} \u(T)\cdot (|C\cap T|-1)\\
 &=& 1+\overbrace{\sum_{T,\, |C\cap T|\geq 2} \delta_{N}(T)\cdot(|C\cap
 T|-1)}^{|C|-1} +
 \sum_{i\in N}\, \sum_{B\subseteq N\setminus\{ i\}} \eta(i|B)\cdot\\
 &&
 \big\{ \sum_{T,\, |C\cap T|\geq 2} \delta_{B}(T)\cdot(|C\cap T|-1)
 - \sum_{T,\, |C\cap T|\geq 2} \delta_{\{ i\}\cup B}(T)\cdot(|C\cap
 T|-1)\,\big\}\\
 &=& |C|+ \sum_{i\in N}\, \sum_{B\subseteq N\setminus\{ i\}}
 \eta(i|B)\cdot \\
 &&\left\{\,\delta(\,|C\cap B|\geq 2\,)\cdot(|C\cap B|-1)
 -\delta(\,|C\cap (\{ i\}\cup B)|\geq 2\,)\cdot(|C\cap (\{ i\}\cup
 B)|-1)\,\right\}.
 \end{eqnarray*}
 \normalsize
 Now, we realize the that the inner expression in braces vanishes
 for $i\not\in C$ because then $C\cap B=C\cap (\{ i\}\cup B)$.
 Analogously, it vanishes if $i\in C$ but $C\cap B=\emptyset$.
 However, in case $i\in C$ and $C\cap B\neq\emptyset$ it equals to
 $-1$. Thus, we write using (\ref{eq.jaa-equal}) for $i\in C$:
 \begin{eqnarray*}
 (\ast ) &=& |C|+ \sum_{i\in N}\, \sum_{B\subseteq N\setminus\{ i\}}
 \eta(i|B)\cdot \delta(\, i\in C \ \& \ C\cap B\neq\emptyset\,)\cdot
 (-1)\\
 &=&  |C|- \sum_{i\in C}\, \sum_{B\subseteq N\setminus\{ i\}}
 \eta(i|B)\cdot \delta(\, C\cap B\neq\emptyset\,)\\
  &\stackrel{(\ref{eq.jaa-equal})}{=}&
  \sum_{i\in C}\, \sum_{B\subseteq N\setminus\{ i\}} \eta(i|B)
  - \sum_{i\in C}\, \sum_{B\subseteq N\setminus\{ i\},\,  B\cap C\neq\emptyset}
 \eta(i|B)\\
  &=&   \sum_{i\in C}\,
  \sum_{B\subseteq N\setminus\{ i\},\,  B\cap C=\emptyset}
  \eta(i|B)\,,
 \end{eqnarray*}
 which gives the third required equality.
 \end{proof}

 \begin{example}\label{exa.7}\rm
 Take $N=\{a,b,c\}$. By (\ref{eq.cluster-2}), there are four
 transformed cluster inequalities for $|C|\geq 2$ breaking into
 two types:
 \begin{itemize}
 \item $\u (\{ a,b\}) +\u (\{ a,b,c\})\geq 0$, \hfill (for $C=\{ a,b\}$)
 \item $\u (\{ a,b\}) +\u (\{ a,c\}) +\u (\{ b,c\}) +2\cdot\u (\{ a,b,c\})\geq 0$.
  \hfill (for $C=\{ a,b,c\}$)
 \end{itemize}
 We observe they coincide with two types of non-specific inequality
 constraints mentioned in Example \ref{exa.3} (see
 \S\,\ref{sssec.our-approx}). Nevertheless, the remaining non-specific
 constrain mentioned there, namely $\u (\{ a,b,c\})\geq 0$, is not
 implied by the transformed cluster inequalities. For instance,
 the $\u$-vector given by $\u(T)=(-1)^{|T|}$ for $T\subseteq \{ a,b,c\}$
 shows that.
 \end{example}

 The above example suggests that the transformed cluster
 inequalities are implied by the non-specific ones, which is indeed
 the case.

 \begin{corollary}\label{cor.cluster-to-nonspec}
 The cluster inequalities transformed to the framework of\/ $\u$-vectors
 (\ref{eq.cluster-2}) follow from non-specific inequality constraints
 (\ref{eq.non-specific-con}).
 \end{corollary}

 \begin{proof}
 By (\ref{eq.cluster-2}), the cluster inequality for $C\subseteq
 N$, $|C|\geq 2$ has the form
 $$
 \langle m_{C},\u\rangle =\sum_{T\subseteq N} m_{C}(T)\cdot\u (T)\geq 0,
 ~~ \mbox{\rm with} ~
 m_{C}(T)=\max\, \{ 0, |C\cap T|-1\} ~ \mbox{for $T\subseteq N$}.
 $$
 The function $m_{C}$ is a special (standardized extreme)
 supermodular function, and, therefore, the inequality for $C$
 follows from (\ref{eq.non-specific-con}).
 \end{proof}

 Thus, we can summarize. The exact translation of the equality
 constraints (\ref{eq.jaa-equal}) to the framework of $\u$-vectors
 are the equality constraints (\ref{eq.standardize}) -- see
 \S\,\ref{sssec.eta-to-stan}. Provided (\ref{eq.jaa-equal}) is valid,
 the exact translation of non-negativity constraints (\ref{eq.jaa-non-neg}) are specific
 inequality constraints (\ref{eq.specific-con}) (see Corollary
 \ref{cor.nonneg-to-speci} in \S\,\ref{sssec.trans-elem-stand}),
 and by Corollary \ref{cor.cluster-to-nonspec}, the cluster
 inequalities (\ref{eq.cluster-1}) translate to some of the
 non-specific inequality constraints (\ref{eq.non-specific-con}).
 In particular, we have

 \begin{corollary}\label{cor.main}
 The $\u$-polyhedron specified by (\ref{eq.standardize})-(\ref{eq.non-specific-con})
 is contained in the image of the $\boldeta$-polyhedron specified by
 (\ref{eq.jaa-non-neg})-(\ref{eq.cluster-1}) by the mapping
 $\boldeta\mapsto\u^{\sboldeta}$ defined in
 (\ref{eq.jaa-to-stand}),
 which is the polyhedron specified by (\ref{eq.standardize}),
 (\ref{eq.specific-con}) and (\ref{eq.cluster-2}).
 \end{corollary}

 \section{LP relaxation}\label{sec.LPrelax}

 To motivate the next result consider the case of three variables
 and transform Jaakkola et al.'s polyhedron $\J$ (\S\,\ref{sssec.jaa-approx})
 to the framework of $\c$-vectors.

 \begin{example}\label{exa.8}\rm
 In Example \ref{exa.5} (see \S\,\ref{ssec.trans-elem}), we transformed
 the elementary constraints (\ref{eq.jaa-non-neg})-(\ref{eq.jaa-equal})
 to the framework of characteristic imsets in case $N=\{ a,b,c\}$.
 The result was a polyhedron given by fifteen inequalities and four equality
 constraints. One can add the transformed cluster inequalities
 (\ref{eq.cluster-3}) to those constraints. There are four such
 inequalities breaking into two types:
 \begin{itemize}
 \item $0\leq 1-\c(ab)$, \hfill (for $C=\{ a,b\}$)
 \item $0\leq 2-\c(ab)-\c(ac)-\c(bc)+\c(abc)$.
 \hfill (for $C=\{ a,b,c\}$)
 \end{itemize}
 We computed (again by {\sl Polymake} \cite{Polymake}) the vertices of
 the resulting polyhedron (= the image of $\J$) and got 12 vertices. The type
 representatives are as follows:
 $$
 [0,0,0,0],
 [1,0,0,0],
 [1,1,0,0],
 [1,1,0,1],
 [1,1,1,1],
 [1,1,1,\frac{3}{2}].
 $$
 All the eleven lattice points here are characteristic imsets
 (for acyclic directed graphs), while the fractional vertex $[1,1,1,\frac{3}{2}]$
 is not. However, it is a convex combination of  vertices of the bigger
 polyhedron (= of the image of $\J^{\prime}$), namely of
 $[2,2,2,3]$ and $[0,0,0,0]$ -- see Example \ref{exa.5}.

 To get the exact polyhedral characterization of the characteristic
 imset polytope (= of the convex hull of the set of characteristic
 imsets) in this case $N=\{ a,b,c\}$ one has to add the translation
 of the non-specific inequality constrain $\u (\{ a,b,c\})\geq 0$
 -- see Example \ref{exa.7}. By (\ref{eq.portrait})-(\ref{eq.characteristic}),
 it leads to
 \begin{itemize}
 \item $\c(abc)\leq 1$,
 \end{itemize}
 which clearly cuts off the fractional vertex and the result is just
 the polytope spanned by the remaining eleven lattice points.
 Thus, the example shows that the basic inequalities for
 characteristic imsets mentioned in \S\,\ref{sssec.advantage}
 (Corollary \ref{cor.basic-char}) are not implied by the transformed
 Jaakkola et al.'s inequalities (\ref{eq.jaa-non-neg})-(\ref{eq.cluster-1}).
 \end{example}

 Nevertheless, we have observed that in case $|N|=3$ the only lattice points
 within the transformed polyhedron are the characteristic imsets. This leads
 to a hypothesis that this holds for any $|N|$. We confirm this
 conjecture now using the observation from Lemma \ref{lem.6}.
 Thus, by transforming Jaakkola et al.'s polyhedron $\J$ we
 get an explicit LP relaxation of the characteristic imset polytope.

 \begin{corollary}\label{cor.char-LP-relax}
 The only lattice points within the polyhedron of $\c$-vectors given by
 (\ref{eq.equal-char}), (\ref{eq.spec-char}) and (\ref{eq.cluster-3})
 are characteristic imsets (for acyclic directed graphs).
 \end{corollary}

 \begin{proof}
 Let us interpret any $\c$-vector as an element of
 ${\dv R}^{{\cal P}_{1}(N)}\equiv {\dv R}^{m}$, that is,
 $\c(\emptyset )=1$ by a convention.
 We have already observed that the polyhedron given by (\ref{eq.equal-char}),
 (\ref{eq.spec-char}) and (\ref{eq.cluster-3}) is the image
 of the polyhedron $\J$ specified by (\ref{eq.jaa-non-neg})-(\ref{eq.cluster-1})
 by the transformation $\boldeta\mapsto\boldB\boldeta=\c$ defined
 in (\ref{eq.eta-to-char}) -- see \S\,\ref{sssec.trans-elem-char}
 and \S\,\ref{ssec.trans-cluster}.

 Assume $\c$ is a lattice point in the considered polyhedron.
 Thus, $\c$ has a pre-image $\boldx\in\J$, which implies that the
 polyhedron $\{\boldx\geq 0;\ \boldB\boldx=\c\}$ in ${\dv R}^{n}$
 is non-empty. By Lemma \ref{lem.6}, $\boldB$ is unimodular, which
 allows us to use Theorem 19.2 in \cite{Sch86} saying that a full row rank
 $m\times n$-matrix $\boldB$ is unimodular if and only if the polyhedron
 $\{\boldx\geq 0;\ \boldB\boldx=\c\}$ is integral for any
 $\c\in {\dv Z}^{m}$. That means, it is the convex hull of its lattice points.
 In particular, since it is non-empty, it has at least one lattice point.
 Let us fix one such lattice point $\boldeta\in {\dv Z}^{n}$,
 $\boldeta\geq 0$ with $\boldB\boldeta =\c$. It automatically
 satisfies (\ref{eq.jaa-non-neg}); (\ref{eq.jaa-equal}) holds
 because $\c (S)=1$ for $|S|=1$ and $\boldB\boldeta=\c$.
 As (\ref{eq.cluster-3}) holds for $\c$, $\boldeta$ satisfies
 all cluster inequalities (\ref{eq.cluster-1}) (by Lemma \ref{lem.7}).
 That means, $\boldeta$ is a lattice point in $\J$.

 By Lemma \ref{lem.2}, $\boldeta$ is necessarily the code $\boldeta_{G}$ of an
 acyclic directed graph $G$ over $N$. By Lemma \ref{lem.3}, its
 image $\c$ by the characteristic transformation is the
 characteristic imset $\c_{G}$ corresponding to $G$.
 \end{proof}

 \noindent {\em Remark}~ In the proof of Corollary \ref{cor.char-LP-relax},
 we have shown that if $\c$ is a lattice point in the cone
 generated by columns of $\boldB$ then it is a non-negative integer
 combination of columns of $\boldB$. That means, in terms of
 \S\,16.4 of \cite{Sch86}, the columns of $\boldB$ form
 the minimal integral Hilbert basis of the cone generated by them.
 Following the terminology from commutative algebra,
 the semigroup generated by columns of $\boldB$ is {\em normal} \cite{Miller2005,Sturmfels1995,Takemura2007}.
 \medskip

 Nevertheless, because of the one-to-one correspondence between
 $\u$-vectors and $\c$-vectors, we have an analogous result
 in the framework of standard imsets.

 \begin{corollary}\label{cor.stan-LP-relax}
 The polyhedron of $\u$-vectors given by (\ref{eq.standardize}),
 (\ref{eq.specific-con}) and (\ref{eq.cluster-2}) is an LP
 relaxation of the standard imset polytope.
 \end{corollary}

 \begin{proof}
 As explained in \S\,\ref{sec:charimset} the mapping $\u\mapsto\c$
 given by (\ref{eq.portrait})-(\ref{eq.characteristic}) is invertible
 and maps lattice points to lattice points. Moreover,
 (\ref{eq.standardize}) is transformed to (\ref{eq.equal-char}),
 (\ref{eq.specific-con}) to (\ref{eq.spec-char}) by Lemma \ref{lem.5} and
 (\ref{eq.cluster-2}) to (\ref{eq.cluster-3}) by Lemma \ref{lem.7}.
 Thus, the image of the polyhedron is the polyhedron of
 $\c$-vectors from Corollary \ref{cor.char-LP-relax}. The
 pre-images of characteristic imsets are standard imsets.
 \end{proof}

 Note that one can also prove Corollary \ref{cor.stan-LP-relax}
 directly, by the method Corollary \ref{cor.char-LP-relax} was
 proved. Indeed, one can use an analogous consideration where
 the matrix $\boldB$ is replaced by $\boldA$ and the vector $\c$
 by $\boldb_{\usc}$ for an $\u$-vector -- see the relation (\ref{eq.problem-2})
 mentioned in \S\,\ref{sssec.trans-elem-stand}.

 Thus, we have an explicit LP relaxation of the standard imset
 polytope and the conjecture from \cite{SV11} is confirmed:

 \begin{corollary}\label{cor.confirm-LP-relax}
 The polyhedron of $\u$-vectors given by
 (\ref{eq.standardize})-(\ref{eq.non-specific-con})
 is an LP relaxation of the standard imset polytope.
 \end{corollary}

 \begin{proof}
 This follows from Corollaries \ref{cor.stan-LP-relax} and
 \ref{cor.main}.
 \end{proof}

 \section*{Conclusions}

 Corollary \ref{cor.char-LP-relax} gives an explicit LP relaxation
 of the characteristic imset polytope. Nevertheless, some of the
 inequalities (\ref{eq.spec-char}) are superfluous because they
 follow from the remaining inequalities. Moreover, perhaps adding
 the basic inequalities from Corollary \ref{cor.basic-char} allows
 one further reduction of the number of inequalities.

 Another research direction is to look for even more loose LP
 relaxation of the standard/characteristic imset polytope, which
 however, has a less number of inequalities.

 \subsection*{Acknowledgments}
 Milan Studen\'{y} was supported by GA \v{C}R grant 201/08/0539.
% David Haws was supported by NIH R01 grant 5R01GM086888. (Rudy suggested I not put this. :) )

 \end{document}